\documentclass[11pt,letterpaper]{article}

\usepackage{graphicx}
\usepackage{psfrag}
\usepackage{enumerate}

\usepackage{amsmath,amssymb,amsthm}
\usepackage{mathrsfs}

\allowdisplaybreaks

\newtheorem{theorem}		{Theorem}[section]

\newtheorem{lemma}		[theorem]{Lemma}
\newtheorem{sideremark}		[theorem]{Remark}

\newenvironment{remark}		{\begin{sideremark}\rm}{\end{sideremark}}
\newcommand{\rev}[1]		{{{#1}}}

\DeclareMathOperator*{\argmax}{arg\,max}
\DeclareMathOperator*{\stat}{stat}
\DeclareMathOperator*{\argstat}{arg\,stat}

\newcommand{\Frechet}		{{Fr\'{e}chet}}

\newcommand{\opdot}[1]		{{\dot{\op{#1}}}}
\newcommand{\ophat}[1]		{{\wh{\op{#1}}}}
\newcommand{\ul}[1]			{{\underline{#1}}}
\newcommand{\cSvex}[1]		{{\mathscr{S}_+^{\op{#1}}}}
\newcommand{\cSave}[1]		{{\mathscr{S}_-^{\op{#1}}}}
\newcommand{\cSv} 		{{\cSvex{-M}}}
\newcommand{\cSa} 		{{\cSave{-M}}}

\newcommand{\cl}			{{\text{cl}}}
\newcommand{\lsc}			{{\text{lsc}}}

\newcommand{\ba}				{\begin{array}}
\newcommand{\ea}				{\end{array}}
\newcommand{\er}[1]			{{(\ref{#1})}}
\newcommand{\nn}				{{\nonumber}}

\newcommand{\ddtone}[2]			{{\frac{d{#1}}{d{#2}}}}
\newcommand{\pdtone}[2]			{{\frac{\partial {#1}}{\partial {#2}}}}

\newcommand{\ts}[1]				{{\textstyle{#1}}}

\newcommand{\dom}			{{\mathsf{dom\, }}}
\newcommand{\funspace}[1]		{{\mathscr{#1}}}
\newcommand{\ol}[1]				{{\overline{#1}}}
\newcommand{\op}[1]			{{\mathcal{#1}}}
\newcommand{\wh}[1]			{{\widehat{#1}}}
\newcommand{\wt}[1]			{{\widetilde{#1}}}

\newcommand{\optilde}[1]			{{\wt{\mathcal{#1}}}}

\newcommand{\N}				{{\mathbb{N}}}
\newcommand{\Q}				{{\mathbb{Q}}}
\newcommand{\R}				{{\mathbb{R}}}
\newcommand{\Z}				{{\mathbb{Z}}}

\newcommand{\cL}				{{\funspace{L}}}
\newcommand{\cU}				{{\funspace{U}}}
\newcommand{\cW}				{{\funspace{W}}}
\newcommand{\cX}				{{\funspace{X}}}
\newcommand{\cY}				{{\funspace{Y}}}

\newcommand{\bo}				{{\mathcal{L}}}
\newcommand{\demi}			{{\ts{\frac{1}{2}}}}
\newcommand{\eps}				{{\epsilon}}
\newcommand{\grad}			{{\nabla}}
\newcommand{\half}				{{\frac{1}{2}}}
\newcommand{\Ltwo}			{{{\cL}_2}}
\newcommand{\ltwo}				{{$\Ltwo$}}
\newcommand{\opA}				{{\Lambda}}
\newcommand{\opAsqrt}			{{\opA^\half}}

\begin{document}


\title{Verifying fundamental solution groups for lossless wave equations via stationary action and optimal control}



\author{Peter M. Dower%
\thanks{%
Department of Electrical \& Electronic Engineering, University of Melbourne, Victoria 3010, Australia.
\tt{pdower@unimelb.edu.au}}
\and
William M. McEneaney%
\thanks{%
Department of Mechanical \& Aerospace Engineering, University of California at San Diego, 9500 Gilman Drive, La Jolla CA 92093-0411, USA.
\tt{wmceneaney@eng.ucsd.edu}}}

\date{}

\maketitle



\begin{abstract}
A representation of a fundamental solution group for a class of wave equations is constructed by exploiting connections between stationary action and optimal control. By using a Yosida approximation of the associated generator, an approximation of the group of interest is represented for sufficiently short time horizons via an idempotent convolution kernel that describes all possible solutions of a corresponding short time horizon optimal control problem. It is shown that this representation of the approximate group can be extended to arbitrary longer horizons via a concatenation of such short horizon optimal control problems, provided that the associated initial and terminal conditions employed in concatenating trajectories are determined via a stationarity rather than an optimality based condition. The long horizon approximate group obtained is shown to converge strongly to the exact group of interest, under reasonable conditions. The construction is illustrated by its application to the solution of a two point boundary value problem.\\[0mm]

\end{abstract}


\section{Introduction}
\label{sec:intro}
The {\em action principle} \cite{F:48,F:64,GT:07,MD1:15} is a variational principle underpinning modern physics that may be applied to a predefined notion of {\em action} to yield the equations of motion of a physical system and its underlying conservation laws. \rev{With a suitable definition of this action, the action principle specialises to {\em Hamilton's action principle}, an important corollary of which states that}
\begin{center}
\vspace{-1mm}
\parbox[t]{14cm}{%
\centering
{\em any trajectory of an energy conserving system renders the corresponding action functional stationary in the calculus of variation sense.
}
}
\vspace{1mm}
\end{center}
\rev{Consequently, Hamilton's action principle} can be interpreted as providing a characterisation of all solutions of \rev{an energy conserving or {\em lossless}} system. This interpretation motivates the development summarised in this work, with \rev{Hamilton's action principle} applied via an optimal control representation to construct the fundamental solution group corresponding to a lossless wave equation. The specific wave equation of interest is given by
\begin{align}
	\ddot x
	& = -\opA\, x\,,
	\label{eq:wave}
\end{align}
where $\opA$ is a linear, unbounded, positive, self-adjoint operator defined on a domain $\cX_2$ dense in an {\ltwo}-space $\cX$, with a compact inverse $\opA^{-1}\in\bo(\cX)$. The results presented generalise the recent work \cite{DM1:17} from the specific \rev{case} $\opA \doteq -\partial^2$ to any unbounded operator $\opA$ satisfying the stated assumptions.

In order to apply \rev{Hamilton's action principle in this development}, compatible notions of kinetic and potential energy are defined with respect to generalised notions of momentum (or velocity) and position that correspond respectively to the input and mild solution of an abstract Cauchy problem \cite{P:83,CZ:95}. This allows the integrated action to be rigorously defined as a time horizon parameterised functional of the momentum (velocity) input. Unlike the finite dimensional case, this action functional is neither convex nor concave for any time horizon, thereby preventing an immediate generalisation of the optimal control approach of \cite{MD1:15} to its analysis. As a remedy, a corresponding approximate class of wave equations is considered as an interim step, in which the unbounded linear operator involved is replaced by its (bounded) Yosida approximation. This yields a corresponding action functional that is strictly concave for sufficiently short (but positive) time horizons. The integrated action is subsequently analysed using tools from optimal control theory, semigroup theory, and idempotent analysis, see for example \cite{P:83,CZ:95,KM:97,DM1:15,DM1:17}. In particular, an {\em idempotent fundamental solution semigroup} applicable on sufficiently short horizons is used to represent the value function of the attendant optimal control problem as an idempotent convolution of a bivariate kernel with a terminal cost. As the characteristics associated with this optimal control problem must correspond to solutions of the approximate wave equation by stationary action, the idempotent fundamental solution semigroup is subsequently used to construct a short horizon prototype for the fundamental solution group for the aforementioned approximate wave equation.  These short horizon prototypes are pieced together into long horizon prototypes using the staticisation operation {\em stat} of \cite{MD1:17}, with the latter converging strongly to the fundamental solution group of the exact wave equation as the Yosida approximation converges strongly to the generator.

In terms of organisation, exact and approximate fundamental solutions groups for the lossless wave equation \er{eq:wave} are first established in Section \ref{sec:groups}. Independently, an optimal control problem that encapsulates \rev{Hamilton's action} principle is introduced in Section \ref{sec:control}.
This control problem is then employed in Section \ref{sec:group} to recover the long-horizon group representation of interest, via a concatenation of short horizon prototypes, thereby confirming the groups of Section \ref{sec:groups}. An application of this representation, and related arguments, to solving a two point boundary value problem (TPBVP) involving \er{eq:wave} is considered in Section \ref{sec:TPBVP}. The paper concludes with some brief remarks in Section \ref{sec:conc}. 
Throughout, $\R$ ($\R_{\ge 0}$) denotes the real (nonnegative) numbers, $\ol{\R}\doteq\R\cup\{\pm\infty\}$,
$\Q$ denotes the rationals, and $\N$ denotes the natural numbers.




\rev{
\section{Exact and approximate fundamental solution groups}
\label{sec:groups}

As $\opA$ is a linear, unbounded, positive, and self-adjoint operator, it possesses a unique, linear, unbounded, positive, and self-adjoint square-root, denoted throughout by $\opAsqrt$. The domain of $\opAsqrt$ defines a Hilbert space, with
\begin{equation}
	\cX_1 \doteq \dom(\opAsqrt) = \ol{\cX}_2\,, 
	\quad\langle x,\xi \rangle_1 \doteq \langle \opAsqrt\, x,\, \opAsqrt\, \xi \rangle\,,
	\label{eq:X-half}
\end{equation}
for all $x,\xi\in\cX_1$, in which $\langle\,, \rangle$ represents the inner product on $\cX$, and $\cX_2\doteq\dom(\opA)$. For convenience, define operator $\op{I}_\mu$ in terms of the resolvent $\op{R}_{-\opA}(\ts{\frac{1}{\mu^2}})$ of $-\opA$ by
\begin{align}
	&
	\op{I}_\mu \doteq \ts{\frac{1}{\mu^2}} \op{R}_{-\opA}(\ts{\frac{1}{\mu^2}}) = (\op{I} + \mu^2\, \opA)^{-1}. 
	\label{eq:op-I-mu}
\end{align}
Boundedness of $\op{I}_\mu$ follows by the Hille-Yosida theorem \cite{P:83}, and in particular it may be identified as an element of $\bo(\cX)$, $\bo(\cX_1)$, or indeed $\bo(\cX;\cX_1)$, etc.


Using definitions \er{eq:X-half} and \er{eq:op-I-mu}, wave equation \er{eq:wave} 
motivates consideration of the linear generators 
\begin{gather}
	\begin{aligned}
	\op{A}
	& \doteq \left( \ba{cc}
		0 & \op{I}
		\\
		-\opA
		& 0
	\ea \right)\!,
	&&
	\dom(\op{A}) \doteq \cY_1\doteq\cX_2\times\cX_1,
	\\
	\op{A}^\mu
	& \doteq \left( \ba{cc}
		0 & \op{I}_\mu^\half
		\\
		-\opA^\half\, \op{I}_\mu^\half\, \opA^\half
		& 0
	\ea \right)\!,
	&&
	\dom(\op{A}^\mu) \doteq \cY\doteq\cX_1\times\cX,
	\end{aligned}
	\label{eq:generators}
\end{gather}
in which it may be noted that $\cY$ defines a Hilbert space with $\langle (x,p), (\xi,\pi) \rangle_\cY \doteq \langle x,\xi \rangle_1 + \langle p,\pi \rangle$ for all $(x,p), (\xi,\pi)\in\cY$, and $\cY_1$ is dense in $\cY$. 
As $\opA$ in \er{eq:wave}, \er{eq:generators} has a compact inverse, the spectral theorem (see for example, \cite[Theorem A.4.25, p.619]{CZ:95}) implies that $\opA$ has the representations
\begin{equation}
	\begin{aligned}
	\opA\, \xi
	& = \sum_{n=1}^\infty \lambda_n \, \langle \xi, \, \tilde\varphi_n \rangle_1\, \tilde\varphi_n\,,
	&& \xi\in\cX_1,
	\\
	\opA\, \pi
	& = \sum_{n=1}^\infty \lambda_n \, \langle \pi, \, \varphi_n \rangle \, \varphi_n\,,
	&& \pi\in\cX,
	\end{aligned}
	\label{eq:spectral-opA}
\end{equation}
on $\cX_1$ and $\cX$ respectively. Here, the set of all eigenvalues $\{\lambda_n^{-1}\}_{n\in\N}$ of compact $\Lambda^{-1}$ defines a strictly positive and strictly decreasing sequence in $\R_{>0}$ \rev{satisfying $0 = \lim_{n\rightarrow\infty} \lambda_n^{-1}$}, while $\{\tilde\varphi_n\}_{n\in\N}$, $\{\varphi_n\}_{n\in\N}$ denote respectively the corresponding sets of orthonormal eigenvectors in $\cX_1$, $\cX$, with $\varphi_n \doteq \sqrt{\lambda_n} \, \tilde\varphi_n$. 

\begin{remark}
\label{rem:op-I-eig}
Given $\mu\in(0,1]$, operators $(\opA\, \op{I}_\mu)^\half$ and $\opA\, \op{I}_\mu\equiv\Lambda^\half\, \op{I}_\mu\, \Lambda^\half$ defined via \er{eq:op-I-mu} and \er{eq:spectral-opA} naturally inherit the spectral form \er{eq:spectral-opA}, see for example the proof of Lemma \ref{lem:op-properties} later. For these specific operators, residing in $\bo(\cX_1)$, the corresponding eigenvalues are given by
\begin{align}	
	\omega_n^\mu
	& \doteq \sqrt{\lambda_n^\mu}\,,
	\quad
	\lambda_n^\mu
	\doteq \frac{\lambda_n}{1 + \mu^2 \, \lambda_n}\,,
	\quad n\in\N,
	\label{eq:omega-lambda}
\end{align}
respectively. 
\hfill{$\square$}
\end{remark}
These definitions and representations imply the following properties of $\op{A}$ and $\op{A}^\mu$ of \er{eq:generators}.
\begin{theorem} \cite{DM1:17,DM3:15,DM4:15}
\label{thm:generators}
Given $\mu\in\R_{>0}$, operators $\op{A}$ and $\op{A}^\mu$ of \er{eq:generators} satisfy the following properties:
\begin{enumerate}[\em (i)]
\item $\op{A}$ is unbounded, closed, and densely defined on $\cY_1\subset\cY$, $\ol{\cY_1} = \cY$;
\item $\op{A}$ generates a strongly continuous group of bounded linear operators $\{ \op{U}_t \}_{t\in\R}\subset\bo(\cY_1)$;
\item $\op{A}^\mu\in\bo(\cY)$;
\item $\op{A}^\mu$ generates a uniformly continuous group of bounded linear operators $\{ \op{U}_t^\mu \}_{t\in\R}\subset\bo(\cY)$, with
\begin{align}
	\op{U}_t^\mu
	& = \left( \ba{c|c}
		[\op{U}_t^\mu]_{11} &
		[\op{U}_t^\mu]_{12} \\[0mm]
		& \\[-2mm]\hline
		& \\[-2mm]
		[\op{U}_t^\mu]_{21} &
		[\op{U}_t^\mu]_{22}
	\ea \right) 	= \exp(t\, \op{A}^\mu), \ t\in\R,
	\nn\\[-5mm]
	& \label{eq:group-mu}	
\end{align}
and $[\op{U}_t^\mu]_{11}\in\bo(\cX_1)$, $[\op{U}_t^\mu]_{12}\in\bo(\cX;\cX_1)$, $[\op{U}_t^\mu]_{21}\in\bo(\cX_1;\cX)$, $[\op{U}_t^\mu]_{22}\in\bo(\cX)$ given by
\begin{align}
	& [\op{U}_t^\mu]_{11} \, \xi
	\doteq
	\sum_{n=1}^\infty \cos(\omega_n^\mu\, t)\, \langle \xi, \tilde\varphi_n \rangle_1\, \tilde\varphi_n,
	& \hspace{-2mm}
	[\op{U}_t^\mu]_{12} \, \pi
	\doteq
	\sum_{n=1}^\infty \sin(\omega_n^\mu\, t)\, \langle \pi, \varphi_n \rangle\, \tilde\varphi_n,
	\nn\\
	& [\op{U}_t^\mu]_{21} \, \xi
	\doteq
	-\!\! \sum_{n=1}^\infty \sin(\omega_n^\mu\, t)\, \langle \xi, \tilde\varphi_n \rangle_1\, \varphi_n,
	& \hspace{-2mm}
	[\op{U}_t^\mu]_{22} \, \pi
	\doteq
	\sum_{n=1}^\infty \cos(\omega_n^\mu\, t)\, \langle \pi, \varphi_n \rangle\, \varphi_n,
	\label{eq:group-mu-elmts}
\end{align}
for all $\xi\in\cX_1$, $\pi\in\cX$. Moreover, there exist $M,\omega\in\R_{\ge 0}$ independent of $\mu$ such that
\begin{align}
	\|\op{U}_s^\mu\|_{\bo(\cY)} 
	& \le M\, \exp(\omega\, s)
	\quad \forall \ s\in\R;
	\label{eq:uniform-semigroup-bound}
\end{align}
\item $\op{A}^\mu$ converges strongly to $\op{A}$ on $\cY_1$ as $\mu\rightarrow 0$, i.e. $\lim_{\mu\rightarrow 0} \|\op{A}^\mu\, y - \op{A}\, y\|_\cY = 0$ for all $y\in\cY_1$;
\item $\op{U}_t^\mu$ converges strongly to $\op{U}_t$, uniformly in $t$ in compact intervals, i.e. $\lim_{\mu\rightarrow 0} \| \op{U}_t^\mu\, y - \op{U}_t\, y \|_\cY = 0$ for all $y\in\cY$, uniformly in $t\in\Omega$, $\Omega\subset\R$ compact.
\end{enumerate}
\end{theorem}

\rev{
\begin{proof}
The proof follows analogously to that of \cite[Lemma 16]{DM1:17} and the details are omitted.
\end{proof}
}


Theorem \ref{thm:generators} and \cite[Theorem 1.3, p.102]{P:83} subsequently imply that there exist unique solutions of the respective abstract Cauchy problems defined via \er{eq:generators} by
\begin{align}
	\left( \ba{c}
		\dot x_s
		\\
		\dot p_s
	\ea \right)
	& = \op{A} \left( \ba{c}
		x_s
		\\
		p_s
	\ea \right)\!,
	& \hspace{-2mm}
	\left( \ba{c}
		\dot \xi_s
		\\
		\dot \pi_s
	\ea \right)
	& = \op{A}^\mu \left( \ba{c}
		\xi_s
		\\
		\pi_s
	\ea \right)\!,
	\ s\in\R,
	\nn\\
	(x_0,p_0) & = (x,p)\in\cY_1,
	& (\xi_0,\pi_0) & = (\xi,\pi)\in\cY\!,
	\nn\\[-5.5mm]
	& 
	& \label{eq:Cauchy}
\end{align}
and that these solutions are continuously differentiable. Hence, $\ddot x_s$ and $\ddot \xi_s$ exist by inspection of \er{eq:generators}, \er{eq:Cauchy}, and satisfy
\begin{align}
	\ddot x_s
	& = -\opA\, x_s\,,
	\quad
	\ddot \xi_s
	= -\opA\, \op{I}_\mu\, \xi_s\,,
	\label{eq:wave-mu}
\end{align}
for all $s\in\R$. That is, \er{eq:wave} holds, as does its approximation obtained from \er{eq:wave} by replacing $-\opA$ with its Yosida approximation $-\opA\, \op{I}_\mu$, see \cite[p.9]{P:83}. Given the groups $\{\op{U}_s\}_{s\in\R}$ and $\{\op{U}_s^\mu\}_{s\in\R}$ generated by $\op{A}$ and $\op{A}^\mu$ as per Theorem \ref{thm:generators}, these solutions necessarily satisfy
\begin{align}
	\left( \ba{c}
		x_s
		\\
		p_s
	\ea \right)
	& = \op{U}_s \left( \ba{c}	
			x
			\\ 
			p
	\ea \right)\!,
	& \hspace{-2mm}
		\left( \ba{c}
		\xi_s
		\\
		\pi_s
	\ea \right)
	& = \op{U}_s^\mu \left( \ba{c}	
			\xi
			\\ 
			\pi
	\ea \right)\!,
	\ s\in\R.
	\nn
\end{align}
The purpose of the analysis that follows is to verify these groups via a construction that appeals only to connections between Hamilton's action principle and optimal control.
}


\section{\rev{Hamilton's action principle} as an optimal control problem}
\label{sec:control}

\rev{In order to apply Hamilton's action principle, the associated action may be defined as} the additive inverse of the integrated Lagrangian \cite{DM1:17,DM3:15,DM4:15}. Explicitly, it is given by
\begin{align}
	&
	\int_0^t V(x_s) - T(p_s) \, ds\,,
	\label{eq:action}
\end{align}
in which $V(x_s)$ and $T(p_s)$ denote the potential and kinetic energies corresponding to a generalized position $x_s$ and momentum $p_s$, at time $s\in\R_{\ge 0}$. 
\rev{To this end, explicitly define the} potential and kinetic energies $V:\cX_1\rightarrow\R_{\ge 0}$ and $T:\cX\rightarrow\R_{\ge 0}$ by
\begin{align}
	\begin{gathered}
	V(x)
	\doteq \demi\, \|x\|_1^2\,,
	\quad
	T(p)
	\doteq \demi\, \|\op{J}\, p\|_1^2 = \demi\, \| p \|^2 \,,
	\end{gathered}
	\label{eq:V-T}
\end{align}
for all $x\in\cX_1$, $p\in\cX$, in which $\op{J} \doteq (\opA^\half)^{-1}$ \rev{and $\|p\|^2 \doteq \langle p,\, p \rangle$}. Note that $\op{J}\in\bo(\cX;\cX_1)$, as $\opA^{-1}$ is bounded.
Stationarity of \er{eq:action} may be encapsulated via an optimal control problem as per \cite{DM1:17,DM3:15,DM4:15}. 


\subsection{Optimal control problem}

With $\cW[0,t]\doteq\Ltwo([0,t];\cX)$, the action \er{eq:action} motivates definition of payoff $J_t:\cX\times\cW[0,t]\rightarrow\R$ for an optimal control problem defined on horizon $t\in\R_{\ge 0}$, with 
\begin{align}
	J_t(x,w) = J_t[\psi](x,w)
	& \doteq \int_0^t V(x_s) - T(w_s) \, ds + \psi(x_t)\,,
	\label{eq:payoff}
	\\
	x_s
	& \doteq x + \int_0^s w_\sigma \, d\sigma\,,
	\quad x\in\cX_1, \ s\in[0,t]\,,
	\label{eq:dynamics}
\end{align}
in which $w\in\cW[0,t]$, and $\psi:\cX_1\rightarrow\ol{\R}$ is 
\rev{a terminal payoff} to be selected later. Unlike finite dimensional problems \cite{MD1:15}, it may be shown that $J_t(x,\cdot)$ need not be concave for any time horizon $t\in\R_{>0}$\rev{, see for example \cite{DM1:17} for a special case.} 
Consequently,
an approximation of \er{eq:payoff} that is concave for sufficiently short, strictly positive, time horizons is first considered. This approximation, denoted by $J_t^\mu:\cX_1\times\cW_1[0,t]\rightarrow\R$, $t\in\R_{>0}$, $\mu\in\R_{>0}$, is defined subject to \er{eq:dynamics} by
\begin{align}
	J_t^\mu(x,w) = J_t^\mu[\psi](x,w)
	& \doteq \int_0^t V(x_s) - T^\mu(w_s)\, ds + \psi(x_t)
	\label{eq:payoff-mu}
\end{align}
for all $x\in\cX_1$, $w\in\cW_1[0,t] \doteq \Ltwo([0,t];\cX_1)$, with $T^\mu:\cX_1\rightarrow\R_{\ge 0}$ approximating $T$ in \er{eq:V-T}, \er{eq:payoff} by
\begin{align}
	T^\mu(w)
	\doteq \demi\, \|w\|^2 + \ts{\frac{\mu^2}{2}}\, \|w\|_1^2 
	& = \demi\, \langle w,\, (\Lambda^{-1} + \mu^2 \, \op{I})\, w \rangle_1
	= \demi\, \langle w,\, \op{I}_\mu^{-1} \, \Lambda^{-1}\, w \rangle_1 
	\nn\\
	& = \demi\, \| (\opA\, \op{I}_\mu)^{-\half}\, w \|_1^2,
	\label{eq:T-mu}		
\end{align}
for all $w\in\cX_1$\rev{, in which $\op{I}_\mu\in\bo(\cX_1)$ is as per \er{eq:op-I-mu}.}
By the asserted properties of $\opA$, $(\opA\, \op{I}_\mu)^{-1} = \opA^{-1} + \mu^2\, \op{I}$ is bounded, positive and self-adjoint on $\cX_1$, and so has a unique bounded positive self-adjoint square-root as per \er{eq:T-mu}.
Similarly, unique $\op{I}_\mu^\half\in\bo(\cX_1)$ exists and commutes with $\opA^\half$ on $\cX_1$. 
As $\cX_1$ is dense in $\cX$, note also that $T^0 \equiv T$.

\begin{lemma} \cite{DM1:17,DM3:15,DM4:15}
\label{lem:concave}
Given \rev{a concave} terminal payoff $\psi:\cX_1\rightarrow\ol{\R}$, the approximate payoff $J_t^\mu(x,\cdot):\cW_1\rightarrow\R_{\ge 0}$ is concave for all $t\in[0,\bar t^\mu)$, and strongly concave for all $t\in(\eps,\bar t^\mu)$, for \rev{any} fixed $\eps\in\R_{>0}$, where $\bar t^\mu \doteq \mu\sqrt{2}$.
\end{lemma}

Given $\mu\in\R_{>0}$, the concavity property provided by Lemma \ref{lem:concave} implies that the value function $W_t^\mu:\cX_1\rightarrow\ol{\R}$ corresponding to \er{eq:payoff-mu} is well-defined for short horizons $t\in[0,\bar t^\mu)$ by
\begin{align}
	W_t^\mu(\xi)
	& \doteq
	\sup_{w\in\cW_1[0,t]} J_t^\mu(\xi,w)\,,
	\label{eq:value-mu}
\end{align}
for all $\xi\in\cX_1$.
%
%
\rev{%
The optimal control problem defined by \er{eq:value-mu} naturally admits a verification theorem, posed with respect to an attendant Hamilton-Jacobi-Bellman (HJB) partial differential equation (PDE), see also \cite[Theorem 6]{DM1:17} for a special case. 
\begin{theorem}
\label{thm:verify}
Given any $\mu\in\R_{>0}$ and $t\in(0,\bar t^\mu)$, suppose there exists a functional $(s,x)\mapsto W_s(x)\in C([0,t]\times\cX_1;\R)\cap C^1((0,t)\times\cX_1;\R)$ such that
\begin{align}
	& 0 = -\pdtone{W_s}{s}(x) + H(x,\grad_x W_s(x))\,,
	\label{eq:verify-HJB}
	\qquad W_0(x) = \psi(x)\,,
\end{align}
for all $s\in(0,t)$, $x\in\cX_1$, where $\grad_x W_s(x)\in\cX_1$ denotes the Riesz representation of the {\Frechet} derivative of $x\mapsto W_s(x)$, defined with respect to the inner product $\langle \cdot\,, \cdot \rangle_1$ on $\cX_1$, and $H:\cX_1\times\cX_1\rightarrow\R$ is the Hamiltonian
\begin{align}
	H(x,p)
	& \doteq \demi\, \|x\|_1^2 + \demi \, \|\op{I}_\mu^\half\, \Lambda^\half\, p\|_1^2
	\label{eq:verify-H} 
\end{align}
for all $x,p\in\cX_1$. Then $W_t(x) \ge J_t^\mu(x,w)$ for all $w\in\cW_1[0,t]$. Furthermore, if there exists a mild solution $s\mapsto\xi_s^*$ as per \er{eq:dynamics} such that
\begin{align}
	& \xi_s^*
	= \xi + \int_0^s w_\sigma^*\, d\sigma,
	\quad
	w_\sigma^*
	= k_\sigma^{\mu,t}(\xi_\sigma^*)\,,
	\label{eq:w-star}
	\\
	& k_\sigma^{\mu,t}(y)
	\doteq \op{I}_\mu^\half\, \op{E}_\mu\, \grad W_{t-\sigma}^\mu(y), \
	\ \sigma\in[0,s]\,, \ s\in[0,t]\,, \ y\in\cX_1\,,
	\nn\\
	& \op{E}_\mu\doteq\opAsqrt\, \op{I}_\mu^\half\, \opAsqrt\in\bo(\cX_1;\cX), \ \op{I}_\mu^\half\in\bo(\cX;\cX_1),
	\label{eq:op-E}
\end{align}
such that $\xi_s^*\in\cX_1$ for all $s\in[0,t]$, then $W_s(x) = J_s^\mu(x,w^*) = W_s^\mu(x)$ for all $s\in[0,t]$, $x\in\cX_1$.
\end{theorem}
\begin{proof}
Fix $\mu\in\R_{>0}$, $t\in(0,\bar t^\mu)$. Let $(s,x)\mapsto W_s(x)\in C([0,t]\times\cX_1;\R)\cap C^1((0,t)\times\cX_1;\R)$ satisfy \er{eq:verify-HJB} as per the theorem statement. Fix any $x\in\cX_1$, $\ol{w}\in\cW_1[0,t]$, and let $s\mapsto\bar\xi_s\in C([0,t];\cX_1)$ denote the corresponding mild solution \er{eq:dynamics} with $\bar\xi_0 = x$ and $w = \ol{w}$. Define $\ol{p}_s \doteq \grad_x W_{t-s}(\bar\xi_s)$ and note that $\ol{p}_s\in\cX_1$ for all $s\in[0,t]$. Fix $s\in[0,t]$. Recalling \er{eq:T-mu}, observe by completion of squares that
\begin{align}
	\langle \ol{p}_s, \, w \rangle_1 - T^\mu(w)
	& = \langle \ol{p}_s, \, w \rangle_1 - \demi\, \| (\Lambda\, \op{I}_\mu)^{-\half}\, w \|_1^2
	\nn\\
	& = \demi\, \| (\Lambda\, \op{I}_\mu)^{\half}\, \ol{p}_s \|_1^2 - \demi\, \| (\Lambda\, \op{I}_\mu)^{-\half}\, [ w - \op{I}_\mu^\half\, \op{E}_\mu\, \ol{p}_s ] \|_1^2
	\nn\\
	& \le \demi\, \| \op{I}_\mu^{\half}\, \Lambda^\half\, \ol{p}_s \|_1^2
	\label{eq:verify-squares}
\end{align}
for all $w\in\cX_1$. Consequently, applying the chain rule and \er{eq:verify-HJB}, \er{eq:verify-H},
\begin{align}
	& \ts{\ddtone{}{s}} [ W_{t-s}(\bar\xi_s) ]
	= -\ts{\pdtone{}{s}} W_{t-s}(\bar\xi_s) + \langle \grad_x W_{t-s}(\bar\xi_s), \, \ol{w}_s \rangle_1
	\nn\\
	& =  \left[ -\ts{\pdtone{}{s}} W_{t-s}(\bar\xi_s) + H(\bar\xi_s,\grad_x W_{t-s}(\bar\xi_s) \right]
	\nn\\
	& \hspace{20mm}
		+ \langle \grad_x W_{t-s}(\bar\xi_s), \, \ol{w}_s \rangle_1
		- \demi\, \|\bar\xi_s\|_1^2 
		- \demi\, \|\op{I}_\mu^\half\, \Lambda^\half\, \grad_x W_{t-s}(\bar\xi_s) \|_1^2
	\nn\\
	& = \langle \grad_x W_{t-s}(\bar\xi_s), \, \ol{w}_s \rangle_1 - \demi\, \|\bar\xi_s\|_1^2 
		- \demi\, \|\op{I}_\mu^\half\, \Lambda^\half\, \ol{p}_s \|_1^2
	\nn\\
	& \le T^\mu(\ol{w}_s) - \demi\, \|\bar\xi_s\|_1^2
	= T^\mu(\ol{w}_s) -V(\bar\xi_s)
	\,,
	\nn
\end{align}
in which the final inequality follows by application of \er{eq:verify-squares} with $w = \ol{w}_s$. Integrating with respect to $s\in[0,t]$ and recalling the initial condition in \er{eq:verify-HJB} subsequently yields
\begin{align}
	\psi(\bar\xi_t) - W_t(x)
	& = 
	\int_0^t \ts{\ddtone{}{s}} [ W_{t-s}(\bar\xi_s) ] \, ds \le 
	\int_0^t T^\mu(\ol{w}_s) -V(\bar\xi_s) \, ds
	\nn\\
	& \Longrightarrow \quad 
	W_t(x)
	\ge \int_0^t V(\bar\xi_s) - T^\mu(\ol{w}_s) \, ds + \psi(\bar\xi_t) = J_t^\mu(x,\ol{w})\,.
	\label{eq:verify-dissipation}
\end{align}
As $x\in\cX_1$ and $\ol{w}\in\cW[0,t]$ are arbitrary, the first assertion follows. Moreover, if $w^*$ exists as per \er{eq:w-star}, selecting $\ol{w}_s \doteq w_s^*$ yields equality in \er{eq:verify-squares}, \er{eq:verify-dissipation}, thereby yielding the second assertion.
\end{proof}
}

\rev{Verification Theorem \ref{thm:verify} is particularly useful in establishing an idempotent convolution representation for the value function \er{eq:value-mu}.}


\subsection{Idempotent convolution representation for \er{eq:value-mu}}
As illustrated in \cite{MD1:15,DMZ1:15,DM1:17}, the value function of an optimal control problem can be expressed as an idempotent convolution of an element of the attendant idempotent fundamental solution semigroup with the terminal payoff of interest. In the specific case of optimal control problem \er{eq:value-mu} for $\mu\in\R_{>0}$, this yields the value function representation
\begin{align}
	W_t^\mu(\xi)
	& = \sup_{\zeta\in\cX_1} \left\{ G_t^\mu(\xi,\zeta) + \psi(\zeta) \right\},
	\label{eq:idem-rep}
\end{align}
for all $t\in(0,\bar t^\mu)$, $\xi\in\cX_1$, in which $G_t^\mu:\cX_1\times\cX_1\rightarrow\ol{\R}$ is the bivariate \rev{idempotent convolution kernel associated with the {\em max-plus primal space fundamental solution semigroup} corresponding to the optimal control problem \er{eq:value-mu}, see for example \cite[Theorem 2]{DM1:17} or \cite[Theorem 5]{D1:19}. Given any $t\in(0,\bar t^\mu)$, this kernel is defined via an optimal TPBVP by
\begin{align}
	G_t^\mu(\xi,\zeta)
	& \doteq \sup_{w\in\cW[0,t]} \left\{ \left. \int_0^t V(x_s) - T^\mu(w_s)\, ds \, \right|  x_0 = \xi, \,  x_t = \zeta \right\}
	\label{eq:G-mu-def}
\end{align}
for all $\xi, \zeta\in\cX_1$. As anticipated by the special case described by \cite[Theorem 11]{DM1:17}, this kernel also has a quadratic representation.
}
\begin{theorem}
\label{thm:G-mu-representation}
Given any $\mu\in\R_{>0}$ and $t\in(0,\bar t^\mu)$, the idempotent convolution kernel $G_t^\mu:\cX_1\times\cX_1\rightarrow\ol{\R}$ of \er{eq:idem-rep}, \er{eq:G-mu-def} has the quadratic representation
\begin{align}
	G_t^\mu(\xi,\zeta)
	& = \demi \left\langle \left( \ba{c}
				\xi \\ \zeta
			\ea \right)\!, \left( \ba{cc}
				\op{P}_t^\mu & \op{Q}_t^\mu
				\\
				\op{Q}_t^\mu & \op{P}_t^\mu
			\ea \right)  \! \left( \ba{c}
				\xi \\ \zeta
			\ea \right)
		\right\rangle_\sharp\!\!,
	\label{eq:G-mu}
\end{align}
for all $\xi, \zeta\in\cX_1$, %
in which $\langle (x,z), (\xi,\zeta) \rangle_\sharp \doteq \langle x,\xi \rangle_1 + \langle z,\zeta\rangle_1$ for all $(x,z), (\xi,\zeta)\in\cX_1\times\cX_1$, and $\op{P}_t^\mu$, $\op{Q}_t^\mu\in\bo(\cX_1)$ are self-adjoint operators. Moreover, these operators also have the spectral form \er{eq:spectral-opA}, with
\begin{align}
	\begin{aligned}
	\op{P}_t^\mu\, \xi
	& \doteq \sum_{n=1}^\infty [p_t^\mu]_n\, \langle \xi,\, \tilde\varphi_n\rangle_1 \, \tilde\varphi_n\,,
	&
	\op{Q}_t^\mu\, \xi
	& \doteq \sum_{n=1}^\infty [q_t^\mu]_n\, \langle \xi,\, \tilde\varphi_n\rangle_1 \, \tilde\varphi_n\,,
	\quad \xi\in\cX_1,
	\end{aligned}
	\label{eq:op-PQR}
\end{align}
in which the respective eigenvalues $[p_t^\mu]_n$, $[q_t^\mu]_n$ are defined by
\begin{align}
	& \hspace{-2mm}
	\lbrack p_t^\mu \rbrack_n
	\doteq \frac{-1}{\omega_n^\mu \, \tan(\omega_n^\mu\, t)},
	\quad
	[q_t^\mu]_n
	\doteq \frac{1}{\omega_n^\mu\, \sin(\omega_n^\mu\, t)},
	\label{eq:eig-pqr}
\end{align}
for all $n\in\N$. 
\end{theorem}
The proof of a special case \cite[Theorem 11]{DM1:17} of Theorem \ref{thm:G-mu-representation} employs a homotopy argument to verify the corresponding quadratic representation analogous to \er{eq:G-mu}. Here, motivated by \cite[Theorem 2]{DZ1:17} and \cite[Theorem 5]{D1:19}, an alternative approach to the proof of Theorem \ref{thm:G-mu-representation} as stated is developed by exploiting semiconvex duality. 
This development commences with the definition of a parameterised terminal cost $\varphi:\cX_1\times\cX_1\rightarrow\R$ by
\begin{align}
	\varphi(x,z)
	& \doteq \demi \, \langle x-z, \, \op{M} \, (x-z) \rangle_1
	\label{eq:basis}
\end{align}
for all $x,z\in\cX_1$, in which $\op{M}\in\bo(\cX_1)$ is a negative self-adjoint operator of spectral form \er{eq:spectral-opA}, with
\begin{align}
	\op{M}\, \xi
	& \doteq \sum_{n=1}^\infty m_n\, \langle \xi,\, \tilde\varphi_n \rangle_1\, \tilde\varphi_n,
	\
	m_n\in[-\ol{m}, \, -\ul{m}],
	\ 0 < \frac{1}{\omega_1^1} \tan \sqrt{2} < \ul{m} \le \ol{m} < \infty, 
	\label{eq:op-M}
\end{align}
for all $\xi\in\cX_1$. Observe by \er{eq:omega-lambda} that
\begin{align}
	\begin{aligned}
	\omega_1^1
	& \le \omega_1^\mu \le \omega_n^\mu \le \omega_\infty^\mu \doteq \frac{1}{\mu}\,, 
	&\qquad& 
	0 < \frac{\mu}{\ol{m}} < 
	\frac{-1}{\omega_n^\mu\, m_n} 
	< \frac{1}{\omega_1^1\, \ul{m}} < \frac{1}{\tan\sqrt{2}}\,,
	\\
	\theta_n^\mu
	& \doteq \tan^{-1}\left( \frac{-1}{ \omega_n^\mu\, m_n} \right),
	&&
	0 < \theta_n^\mu < \tan^{-1} \left( \frac{1}{\tan\sqrt{2}} \right) = \frac{\pi}{2} - \sqrt{2}\,.
	\end{aligned}
	\label{eq:omega-mu-theta}
\end{align}
for all $\mu\in(0,1]$, $n\in\N$. Note in particular that
\begin{align}
	\mu\in(0,1], \ t\in(0,\bar t^\mu)
	& \quad \Longrightarrow \quad
	\omega_n^\mu\, t + \theta_n^\mu
	\in\left( 0, \, \frac{\pi}{2} \right)
	\quad \forall \ n\in\N.
	\label{eq:angle-bounds}
\end{align}
Given $\mu\in(0,1]$, $t\in(0,\bar t^\mu)$, it is useful to define an auxiliary optimal control problem with value function $S_t^\mu:\cX_1\times\cX_1\rightarrow\ol{\R}$ by
\begin{align}
	S_t^\mu(\xi,\zeta)
	& \doteq
	\sup_{w\in\cW_1[0,t]} J_t^\mu[\varphi(\cdot,\zeta)](\xi,w)
	\label{eq:value-S}
\end{align}
for all $\xi,\zeta\in\cX_1$, in which $J_t^\mu$ and $\varphi$ are as per \er{eq:payoff-mu} and \er{eq:basis}. 

\begin{lemma}
\label{lem:value-S-explicit}
Given $\mu\in(0,1]$, $t\in(0,\bar t^\mu)$, the value function $S_t^\mu:\cX_1\times\cX_1\rightarrow\R$ of \er{eq:value-S} has the explicit quadratic representation 
\begin{align}
	S_t^\mu(\xi,\zeta)
	& \doteq \demi\, \langle \xi, \, \op{X}_t^{\mu}\, \xi \rangle_1 + \langle \xi, \, \op{Y}_t^\mu\, \zeta \rangle_1 
			+ \demi\, \langle \zeta, \, \op{Z}_t^\mu\, \zeta \rangle_1
	\label{eq:value-S-explicit}
\end{align}
for all $\xi,\zeta\in\cX_1$, in which $\op{X}_t^\mu, \op{Y}_t^\mu, \op{Z}_t^\mu\in\bo(\cX_1)$ are bounded linear operators of the spectral form \er{eq:spectral-opA}, with respective eigenvalues given by
\begin{align}
	\begin{aligned}
	\lbrack x_t^\mu \rbrack_n 
	& \doteq -\frac{1}{\omega_n^\mu}\, \cot \left( \omega_n^\mu\, t + \theta_n^\mu \right),
	\\
	[y_t^\mu]_n
	& \doteq + \frac{1}{\omega_n^\mu} \, \cos\theta_n^\mu\, \csc(\omega_n^\mu\, t + \theta_n^\mu),
	\\
	[z_t^\mu]_n
	& \doteq -\frac{1}{\omega_n^\mu} \, \cos^2 \theta_n^\mu \left[ \, \tan\theta_n^\mu + \cot(\omega_n^\mu \, t + \theta_n^\mu) \, \right],
	\end{aligned}
	\label{eq:eig-XYZ}
\end{align}
for all $n\in\N$, and satisfying $\opdot{X}_t^\mu, \opdot{Y}_t^\mu, \opdot{Z}_t^\mu\in\bo(\cX_1)$.
\end{lemma}

\begin{proof}
Fix $\mu\in(0,1]$, $t\in(0,\bar t^\mu)$. Let $\op{X}_t^\mu, \op{Y}_t^\mu, \op{Z}_t^\mu$ be linear operators of the spectral form \er{eq:spectral-opA} as per the lemma statement, and observe that their respective eigenvalues satisfy
\begin{align}
	%
	\lbrack x_t^\mu \rbrack_n 
	& 
	\doteq -\frac{1}{\omega_n^\mu}\, \cot \left( \omega_n^\mu\, t + \theta_n^\mu \right),
	\nn\\
	& = m_n + \frac{1}{\omega_n^\mu}\, [ \, -\omega_n^\mu\, m_n - \cot(\omega_n^\mu\, t + \theta_n^\mu) \, ]
	\nn\\
	& = m_n + \frac{1}{\omega_n^\mu}\, [ \, \cot\theta_n^\mu - \cot(\omega_n^\mu\, t + \theta_n^\mu) \, ],
	\nn\\
	[y_t^\mu]_n
	& 
	\doteq + \frac{1}{\omega_n^\mu} \, \cos\theta_n^\mu\, \csc(\omega_n^\mu\, t + \theta_n^\mu)
	\nn\\
	& = -m_n + \frac{1}{\omega_n^\mu} \, \cos\theta_n^\mu\, [ \, \omega_n^\mu\, m_n\, \sec\theta_n^\mu + \csc(\omega_n^\mu\, t + \theta_n^\mu) \, ]
	\nn\\
	& = -m_n + \frac{1}{\omega_n^\mu} \, \cos\theta_n^\mu\, [ \, -\cot \theta_n^\mu\,  \sec\theta_n^\mu + \csc(\omega_n^\mu\, t + \theta_n^\mu) \, ]
	\nn\\
	& =  -m_n - \frac{1}{\omega_n^\mu} \, \cos\theta_n^\mu\, [ \, \csc\theta_n^\mu - \csc(\omega_n^\mu\, t + \theta_n^\mu) \, ]
	\label{eq:eig-XYZ-alt}
	\\
	[z_t^\mu]_n
	& 
	\doteq -\frac{1}{\omega_n^\mu} \, \cos^2 \theta_n^\mu \left[ \, \tan\theta_n^\mu + \cot(\omega_n^\mu \, t + \theta_n^\mu) \, \right]
	\nn\\
	& = m_n + \frac{1}{\omega_n^\mu} \, \cos^2 \theta_n^\mu \left[ \, 
			-\omega_n^\mu\, m_n\, \sec^2 \theta_n^\mu - \tan\theta_n^\mu - \cot(\omega_n^\mu \, t + \theta_n^\mu) \, \right]
	\nn\\
	& = m_n + \frac{1}{\omega_n^\mu} \, \cos^2 \theta_n^\mu \left[ \, 
			\cot\theta_n^\mu \, (\sec^2 \theta_n^\mu - \tan^2\theta_n^\mu) - \cot(\omega_n^\mu \, t + \theta_n^\mu) \, \right]
	\nn\\
	& 
	= m_n + \frac{1}{\omega_n^\mu} \, \cos^2 \theta_n^\mu \left[ \, \cot\theta_n^\mu - \cot(\omega_n^\mu\,t + \theta_n^\mu) \, \right],
	\nn
\end{align}
for all $n\in\N$. Bounds \er{eq:omega-mu-theta}, \er{eq:angle-bounds} imply that the corresponding eigenvalue sequences are bounded. Moreover, elements of these eigenvalue sequences are differentiable with respect to $t$, and satisfy 
\begin{align}
	&
	\begin{aligned}
	\lbrack \dot x_t^\mu \rbrack_n 
	& 
	= 1 + \lambda_n^\mu \, [x_t^\mu]_n^2,
	&
	[ \ddot x_t^\mu ]_n
	& = 2\, \lambda_n^\mu\, [x_t^\mu]_n \, [\dot x_t^\mu]_n,
	&
	[ x_0^\mu ]_n
	& = m_n,
	\\
	[ \dot y_t^\mu]_n 
	& = \lambda_n^\mu\, [x_t^\mu]_n \, [y_t^\mu]_n,
	&
	[ \ddot y_t^\mu ]_n
	& = \lambda_n^\mu\, ( [\dot x_t^\mu]\, [y_t^\mu]_n + [x_t^\mu]\, [\dot y_t^\mu]_n),
	&
	[ y_0^\mu ]_n & = -m_n,
	\\
	[\dot z_t^\mu]_n
	& = \lambda_n^\mu\, [ y_t^\mu]_n^2,
	&
	[\ddot z_t^\mu]_n
	& = 2\, \lambda_n^\mu\, [y_t^\mu]_n \, [\dot y_t^\mu]_n,
	& 
	[z_0^\mu]_n 
	& = m_n,
	\end{aligned}
	\label{eq:eig-derivatives}
\end{align}
so that the sequences of corresponding derivatives are also bounded. Hence, the linear operators $\op{X}_t^\mu, \op{Y}_t^\mu, \op{Z}_t^\mu$ are bounded and {\Frechet} differentiable with bounded derivatives $\opdot{X}_t^\mu, \opdot{Y}_t^\mu, \opdot{Z}_t^\mu$. That is,
$\op{X}_t^\mu, \op{Y}_t^\mu, \op{Z}_t^\mu\in\bo(\cX_1)$ and $\opdot{X}_t^\mu, \opdot{Y}_t^\mu, \opdot{Z}_t^\mu\in\bo(\cX_1)$, as asserted. 
By inspection of \er{eq:spectral-opA}, \er{eq:eig-derivatives}, and Remark \ref{rem:op-I-eig}, note further that these operators satisfy the respective Cauchy problems
\begin{align}
	& 
	\begin{aligned}
	\opdot{X}_t^\mu
	& = \op{I} + \op{X}_t^\mu\, \Lambda^\half\, \op{I}_\mu\, \Lambda^\half\, \op{X}_t^\mu,
	&& \op{X}_0^\mu = +\op{M},
	\\
	\opdot{Y}_t^\mu
	& = \op{X}_t^\mu\, \Lambda^\half\, \op{I}_\mu\, \Lambda^\half\, \op{Y}_t^\mu,
	&& \op{Y}_0^\mu = -\op{M},
	\\
	\opdot{Z}_t^\mu
	& = \op{Y}_t^\mu\, \Lambda^\half\, \op{I}_\mu\, \Lambda^\half\, \op{Y}_t^\mu,
	&& \op{Z}_0^\mu = +\op{M},
	\end{aligned}
	\label{eq:op-DREs}
\end{align}
in which $\op{I}\in\bo(\cX_1)$ denotes the identity.

Define $\wh{S}_t^\mu:\cX_1\times\cX_1\rightarrow\R$ as per the quadratic form in the lemma statement, i.e.
\begin{align}
	\wh{S}_t^\mu(\xi,\zeta)
	& \doteq \demi\, \langle \xi, \, \op{X}_t^{\mu}\, \xi \rangle_1 + \langle \xi, \, \op{Y}_t^\mu\, \zeta \rangle_1 
			+ \demi\, \langle \zeta, \, \op{Z}_t^\mu\, \zeta \rangle_1
	\label{eq:value-S-hat-explicit}
\end{align}
for all $\xi,\zeta\in\cX_1$. Observe that $t\mapsto S_t^\mu(\xi,\zeta)$ is {\Frechet} differentiable, as is $\xi\mapsto S_t^\mu(\xi,\zeta)$ by inspection, with  respective derivatives given via their Riesz representation by
\begin{align}
	\begin{aligned}
	\pdtone{\wh{S}_t^\mu}{t}(\xi,\zeta)
	& = \demi\, \langle \xi, \, \opdot{X}_t^{\mu}\, \xi \rangle_1 + \langle \xi, \, \opdot{Y}_t^\mu\, \zeta \rangle_1 
			+ \demi\, \langle \zeta, \, \opdot{Z}_t^\mu\, \zeta \rangle_1,
	\\
	\grad_\xi \wh{S}_t^\mu(\xi,\zeta)
	& = \op{X}_t^\mu\, \xi + \op{Y}_t^\mu\, \zeta,
	\end{aligned}
	\label{eq:value-S-derivatives}
\end{align}
for all $\xi,\zeta\in\cX_1$. 
Consequently, applying \er{eq:op-DREs}, \er{eq:value-S-derivatives} in \er{eq:verify-HJB}, \er{eq:verify-H} yields
\begin{align}
	& -\pdtone{S_t^\mu}{t}(\xi,\zeta) + H(\xi, \grad_\xi S_t^\mu(\xi,\zeta))
	\nn\\
	& = - \demi\, \langle \xi, \, \opdot{X}_t^{\mu}\, \xi \rangle_1 - \langle \xi, \, \opdot{Y}_t^\mu\, \zeta \rangle_1 
			- \demi\, \langle \zeta, \, \opdot{Z}_t^\mu\, \zeta \rangle_1
	\nn\\
	& \hspace{20mm}
	+ \demi\, \|\xi\|_1^2 + \demi\, \langle \op{X}_t^\mu\, \xi + \op{Y}_t^\mu\, \zeta, \, 
			\Lambda^\half\, \op{I}_\mu\, \Lambda^\half\, \left( \op{X}_t^\mu\, \xi + \op{Y}_t^\mu\, \zeta \right) \rangle_1
	\nn\\
	& = \demi\, \langle \xi, (-\opdot{X}_t^\mu + \op{I} +  \op{X}_t^\mu\, \Lambda^\half\, \op{I}_\mu\, \Lambda^\half\, \op{X}_t^\mu) \, \xi \rangle_1
	+ \langle \xi, \, (-\opdot{Y}_t^\mu + \op{X}_t^\mu\, \Lambda^\half\, \op{I}_\mu\, \Lambda^\half\, \op{Y}_t^\mu)\, \zeta \rangle_1
	\nn\\
	& \hspace{20mm}
	+ \demi\, \langle \zeta, \, (-\opdot{Z}_t^\mu + \op{Y}_t^\mu\, \Lambda^\half\, \op{I}_\mu\, \Lambda^\half\, \op{Y}_t^\mu)\, \zeta \rangle_1
	\nn\\
	& = 0
	\nn
\end{align}
in which the final equality follows by \er{eq:op-DREs}. Meanwhile, the initial data in \er{eq:op-DREs} and \er{eq:basis}, \er{eq:value-S-hat-explicit} yield
\begin{align}
	\wh{S}_0^\mu(\xi,\zeta)
	& = \demi\, \langle \xi, \, \op{M} \, \xi \rangle_1 - \langle\xi,\, \op{M}\, \zeta\rangle_1 + \demi\, \langle \zeta,\, \op{M}\, \zeta \rangle_1
	= \varphi(\xi,\zeta).
	\nn
\end{align}
That is, $\wh{S}_t^\mu(\cdot,\zeta):\cX_1\rightarrow\R$, $\zeta\in\cX_1$, satisfies the HJB PDE \er{eq:verify-HJB}, \er{eq:verify-H}. Moreover, mild solution $s\mapsto\xi_s^*$ of \er{eq:w-star} may also be shown to exist, with $\xi_s\in\cX_1$ for all $s\in[0,t]$, using a fixed point argument. The details parallel \cite[Theorem 13]{DM1:17}, and are omitted. Hence, the conditions of verification Theorem \ref{thm:verify} are satisfied, so that $\wh{S}_t^\mu(\cdot,\zeta)$ of \er{eq:value-S-hat-explicit} is the value of an optimal control problem of the form \er{eq:value-mu} with terminal cost $\psi \doteq \varphi(\cdot,\zeta)$. That is, recalling \er{eq:value-S}, 
\begin{align}
	\wh{S}_t^\mu(\xi,\zeta)
	& = 
	\sup_{w\in\cW_1[0,t]} J_t^\mu[\varphi(\cdot,\zeta)](\xi,w)
	= S_t^\mu(\xi,\zeta)
	\label{eq:value-S-equiv}
\end{align}
for all $\xi,\zeta\in\cX_1$ and $t\in(0,\bar t^\mu)$, as required. 
\end{proof}

Coercivity of $\op{Z}_t^\mu - \op{M}$ is useful for the subsequent semiconvex duality argument.

\begin{lemma}
\label{lem:op-Z-M-coercive}
Given $\mu\in(0,1]$, $t\in(0,\bar t^\mu)$, the operator $\op{Z}_t^\mu - \op{M}\in\bo(\cX_1)$ is coercive.
\end{lemma}
\begin{proof}
Fix $\mu\in(0,1]$, $t\in(0,\bar t^\mu)$. The asserted boundedness, i.e. $\op{Z}_t^\mu - \op{M}\in\bo(\cX_1)$, is immediate by \er{eq:op-M} and Lemma \ref{lem:value-S-explicit}. Moreover, this operator has the spectral form \er{eq:spectral-opA}, see \er{eq:op-M}, \er{eq:eig-XYZ}, with
\begin{align}
	\langle \zeta,\, (\op{Z}_t^\mu - \op{M})\, \zeta \rangle_1
	& = \sum_{n=1}^\infty ([z_t^\mu]_n - m_n) \, |\langle \zeta,\, \tilde\varphi_n \rangle_1|^2
	\label{eq:pre-coercive}
\end{align}
Recalling the last equality in \er{eq:eig-XYZ-alt},
\begin{align}
	[z_t^\mu]_n - m_n
	& =  \, [\cos^2 \theta_n^\mu] \, f_n^\mu(t),
	\nn
\end{align}
with
\begin{gather}
	\cos^2\theta_n^\mu\ge\cos^2\left(\frac{\pi}{2} - \sqrt{2}\right) = \sin^2\sqrt{2} > 0,
	\nn\\
	f_n^\mu(t) 
	\doteq \frac{1}{\omega_n^\mu} [\cot\theta_n^\mu - \cot(\omega_n^\mu\, t + \theta_n^\mu)],
	\qquad
	f_n^\mu(0)
	= 0,
	\\
	(f_n^\mu)'(t)
	= \csc^2(\omega_n^\mu\, t + \theta_n^\mu) > 1,
	\nn
\end{gather}
for all $n\in\N$. Consequently, $[z_t^\mu]_n - m_n\ge t\, \sin^2 \sqrt{2}$ for all $n\in\N$, so that by \er{eq:pre-coercive},
\begin{align}
	\langle \zeta,\, (\op{Z}_t^\mu - \op{M})\, \zeta \rangle_1
	& \ge \sum_{n=1}^\infty t\, \sin^2 \sqrt{2}\, |\langle \zeta,\, \tilde\varphi_n \rangle_1|^2
	= t\, \sin^2 \sqrt{2}\, \| \zeta\|_1^2,
	\nn
\end{align}
for all $\zeta\in\cX_1$. That is, $\op{Z}_t^\mu - \op{M}$ is coercive, as required.
\end{proof}

In continuing the preparations for the proof of Theorem \ref{thm:G-mu-representation}, some definitions relating to semiconvex duality are required.
In particular, a function $\psi:\cX_1\rightarrow\ol{\R}$ is convex if its epigraph $\{ (x,\alpha)\in\cX_1\times\R \, | \, \psi(x)\le\alpha \}$ is convex \cite{R:74}. It is {\em lower closed} if $\psi = \cl^-\, \psi$, in which $\cl^-$ is the {\em lower closure} defined with respect to the lower semicontinuous envelope $\lsc$ by
\begin{align}
	\cl^- \,\psi(x)
	& \doteq \left\{ \ba{cl}
		\lsc \, \psi(x),
		& \lsc\, \psi(x) > -\infty \text{ for all } x\in\cX_1,
		\\
		-\infty,
		& \text{otherwise},
	\ea \right.
	\nn
\end{align}
for all $x\in\cX_1$. Similarly, $\psi$ is concave if $-\psi$ is convex, and {\em upper closed} if $-\psi$ is lower closed, see \cite[pp.15-17]{R:74}.
%
%
Following \cite{DM1:15,D1:19}, uniformly semiconvex and semiconcave extended real valued function spaces $\cSv$ and $\cSa$ are defined with respect to operator $\op{M}$ of \er{eq:basis},\er{eq:op-M} by
\begin{align}
	\begin{aligned}
	\cSv
	& \doteq \left\{ \psi:\cX_1\rightarrow\ol{\R} \, \left| \, \ba{c} 
										\psi+\demi\, \langle \cdot\,, -\op{M} \, \cdot \rangle_1
										\\
										\text{convex, lower closed}
									\ea \right. \right\},
	\\
	\cSa
	& \doteq \left\{ \psi:\cX_1\rightarrow\ol{\R} \, \left| \, \ba{c} 
										\psi-\demi\, \langle \cdot\,, -\op{M} \, \cdot \rangle_1
										\\
										\text{ concave, upper closed}
									\ea \right. \right\}.
	\end{aligned}
	\label{eq:semiconvex}
\end{align}
Semiconvex duality is a duality between the spaces of \er{eq:semiconvex}, defined via the semiconvex transform. The semiconvex transform is a generalization of the Legendre-Fenchel transform, in which convexity is weakened to semiconvexity by relaxing affine support to quadratic support. The quadratic support functions involved are defined here via the bivariate quadratic basis function $\varphi:\cX_1\times\cX_1\rightarrow\R$ given by \er{eq:basis}.
The semiconvex transform and its inverse, denoted by  $\op{D}_\varphi:\cSv\rightarrow\cSa$ and $\op{D}_\varphi^{-1}:\cSa\rightarrow\cSv$, are given by \cite{DM1:15,D1:19}
\begin{align}
	\op{D}_\varphi\, \Psi
	& \doteq -\sup_{\xi\in\cX_1} \left\{ \varphi(\xi,\cdot) -\Psi(\xi)\right\},
	\qquad
	\op{D}_\varphi^{-1}\, a
	\doteq \sup_{z\in\cX_1}\left\{ \varphi(\cdot,z) + a(z)\right\},
	\label{eq:op-D}
\end{align}
for all $\Psi\in\cSv$ and $a\in\cSa$. It is also useful to define
\begin{align}
	& \delta^-(\xi,\zeta) \doteq \left\{ \ba{cl} 
			0 & \|\xi-\zeta\|_1 = 0,
			\\
			-\infty & \|\xi-\zeta\|_1\ne 0,
			\ea \right.
	\label{eq:delta-minus}
\end{align}
for all $\xi\in\cX_1$.

These definitions and concepts may now be used to establish a representation for the convolution kernel $G_t^\mu$ of \er{eq:idem-rep}.

\begin{lemma}
\label{lem:kernel-G-idem}
Given $\mu\in(0,1]$, $t\in(0,\bar t^\mu)$, the auxiliary value function $S_t^\mu$ of \er{eq:value-S}, \er{eq:value-S-explicit} and the convolution kernel $G_t^\mu$ of \er{eq:idem-rep}, \er{eq:G-mu-def} satisfy
\begin{gather}
	S_t^\mu(\xi,\cdot)\in\cSv,
	\qquad
	G_t^\mu(\xi,\zeta)
	= [\op{D}_\varphi\, S_t^\mu (\xi,\cdot)](\zeta)
	\label{eq:kernel-G-idem}
\end{gather}
for all $\xi,\zeta\in\cX_1$, in which $\op{D}_\varphi$ is the semiconvex dual operation of \er{eq:op-D} with respect to $\varphi$ of \er{eq:basis}.
\end{lemma}
\begin{proof}
Fix $\mu\in(0,1]$, $t\in(0,\bar t^\mu)$, and $\xi,\zeta\in\cX_1$. Applying \er{eq:basis} and Lemma \ref{lem:value-S-explicit},
\begin{align}
	S_t^\mu(\xi,\zeta) + \demi\, \langle\zeta, \, - \op{M}\, \zeta\rangle_1 
	& 
	= \demi\, \langle \xi, \, \op{X}_t^{\mu}\, \xi \rangle_1 + \langle \xi, \, \op{Y}_t^\mu\, \zeta \rangle_1 
			+ \demi\, \langle \zeta, \, (\op{Z}_t^\mu - \op{M})\, \zeta \rangle_1.
	\nn
\end{align}
As $Z_t^\mu - \op{M}$ is coercive by Lemma \ref{lem:op-Z-M-coercive}, it follows immediately that $\zeta\mapsto S_t^\mu(\xi,\zeta) + \demi\, \langle\zeta, \, - \op{M}\, \zeta\rangle_1$ is convex. Hence, $S_t^\mu(\xi,\cdot)\in\cSv$ by \er{eq:semiconvex}, yielding the first assertion in \er{eq:kernel-G-idem}.

For the remaining assertion in \er{eq:kernel-G-idem}, note by \er{eq:payoff-mu}, \er{eq:basis}, \er{eq:value-S}, and \er{eq:op-D}, that
\begin{align}
	S_t^\mu(\xi,\zeta)
	& = \!\!\! \sup_{w\in\cW_1[0,t]} \!\! J_t^\mu[\varphi(\cdot,\zeta)](\xi,w)
	= \!\!\! \sup_{w\in\cW_1[0,t]} \! \left\{ \int_0^t V(\xi_s) - T^\mu(w_s)\, ds + \varphi(\xi_t, \zeta) \right\}
	\nn\\
	& = \sup_{w\in\cW_1[0,t]} \left\{ \int_0^t V(\xi_s) - T^\mu(w_s)\, ds + \sup_{y\in\cX_1} \left\{ \delta^-(\xi_t,y) + \varphi(y,\zeta) \right\} \right\}
	\nn\\
	& = \sup_{y\in\cX_1} \left\{ \sup_{w\in\cW_1[0,t]} \left\{ \int_0^t V(\xi_s) - T^\mu(w_s)\, ds + \delta^-(\xi_t,y) \right\} + \varphi(y,\zeta) \right\}
	\nn\\
	& = \sup_{y\in\cX_1} \left\{ G_t^\mu(\xi,y) + \varphi(y,\zeta) \right\}
	= \sup_{y\in\cX_1} \left\{ \varphi(\zeta,y) + G_t^\mu(\xi,y) \right\}
	\nn\\
	& = [\op{D}_\varphi^{-1}\, G_t^\mu(\xi,\cdot)](\zeta),
	\nn
\end{align}
in which $\delta^-$ is as per \er{eq:delta-minus}, and the second last equality follows by symmetry of $\varphi$, i.e. $\varphi(\xi,\zeta) = \varphi(\zeta,\xi)$. Hence, by semiconvex duality and the first assertion,
\begin{align}
	G_t^\mu(\xi,\zeta)
	& = [\op{D}_\varphi\, \op{D}_\varphi^{-1} \, G_t^\mu(\xi,\cdot)](\zeta) = [\op{D}_\varphi\, S_t^\mu(\xi,\cdot)](\zeta),
	\nn
\end{align}
yielding the second assertion.
\end{proof}

It remains to prove Theorem \ref{thm:G-mu-representation}, using Lemma \ref{lem:kernel-G-idem}.

\begin{proof}
{\em [Theorem \ref{thm:G-mu-representation}]}
Fix $\mu\in(0,1]$, $t\in(0,\bar t^\mu)$, $\xi,\zeta\in\cX_1$. Applying Lemma \ref{lem:kernel-G-idem},
\begin{align}
	& G_t^\mu(\xi,\zeta)
	= [\op{D}_\varphi \, S_t^\mu(\xi,\cdot)](\zeta)
	= \inf_{y\in\cX_1} \left\{ S_t^\mu(\xi,y) - \varphi(y,\zeta) \right\}
	\nn\\
	& = \inf_{y\in\cX_1} \left\{
		\demi\, \langle \xi, \, \op{X}_t^{\mu}\, \xi \rangle_1 + \langle \xi, \, \op{Y}_t^\mu\, y \rangle_1 
			+ \demi\, \langle y, \, \op{Z}_t^\mu\, y \rangle_1 - \demi\, \langle y-\zeta,\, \op{M}\, (y - \zeta) \rangle_1
	\right\}
	\nn\\
	& = \demi\, \langle \xi, \, \op{X}_t^{\mu}\, \xi \rangle_1 - \demi\, \langle \zeta, \, \op{M}\, \zeta \rangle_1
	\nn\\
	& \hspace{20mm}
	+ \inf_{y\in\cX_1} \left\{ \langle y,\, (\op{Y}_t^\mu)' \, \xi + \op{M}\, \zeta \rangle_1 
				+ \demi\, \langle y,\, (\op{Z}_t^\mu - \op{M})\, y \rangle_1
	\right\}.
	\nn
\end{align}
Applying Lemma \ref{lem:op-Z-M-coercive}, observe that $\op{Z}_t^\mu - \op{M}$ is coercive, and hence boundedly invertible. Consequently, the infimum is achieved at $y = y^*\in\cX_1$, with $y^* \doteq -(\op{Z}_t^\mu - \op{M})^{-1} \, [ (\op{Y}_t^\mu)' \, \xi + \op{M}\, \zeta ]$. By substitution,
\begin{align}
	& G_t^\mu(\xi,\zeta)
	= \demi\, \langle \xi, \, \op{X}_t^{\mu}\, \xi \rangle_1 - \demi\, \langle \zeta, \, \op{M}\, \zeta \rangle_1
	+ \langle y^*,\, (\op{Y}_t^\mu)' \, \xi + \op{M}\, \zeta \rangle_1 
	\nn\\
	& \hspace{30mm}
				+ \demi\, \langle y^*,\, (\op{Z}_t^\mu - \op{M})\, y^* \rangle_1
	\nn\\
	& = \demi\, \langle \xi, \, [\op{X}_t^{\mu} - \op{Y}_t^\mu\, (\op{Z}_t^\mu - \op{M})^{-1} \, (\op{Y}_t^\mu)' ] \, \xi \rangle_1
		- \langle \xi, \, \op{Y}_t^\mu\, (\op{Z}_t^\mu - \op{M})^{-1} \, \op{M}\, \zeta \rangle_1
	\nn\\
	& \qquad\qquad
		+ \demi\, \langle \zeta,\, [-\op{M} - \op{M}\, (\op{Z}_t^\mu - \op{M})^{-1} \, \op{M}]\, \zeta \rangle_1
	\nn\\
	& \doteq \demi\, \langle \xi, \, \ophat{X}_t^\mu \, \xi \rangle + \langle \xi, \, \ophat{Y}_t^\mu \, \zeta \rangle_1 
				+ \demi\, \langle \zeta,\, \ophat{Z}_t^\mu\, \zeta \rangle_1
	= \demi \left\langle \left( \ba{c} \xi \\ \zeta \ea \right), \, \left( \ba{cc} 
						\ophat{X}_t^\mu & \ophat{Y}_t^\mu
						\\
						(\ophat{Y}_t^\mu)' & \ophat{Z}_t^\mu
					\ea \right) \left( \ba{c} \xi \\ \zeta \ea \right)
				\right\rangle_\sharp\!\!,
	\label{eq:pre-kernel-G}
\end{align}
in which $\ophat{X}_t^\mu, \ophat{Y}_t^\mu, \ophat{Z}_t^\mu\in\bo(\cX_1)$ are defined by
\begin{gather}
	\ophat{X}_t^\mu
	\doteq \op{X}_t^{\mu} - \op{Y}_t^\mu\, (\op{Z}_t^\mu - \op{M})^{-1} \, (\op{Y}_t^\mu)',
	\quad
	\ophat{Y}_t^\mu
	\doteq - \op{Y}_t^\mu\, (\op{Z}_t^\mu - \op{M})^{-1} \, \op{M},
	\quad
	 \nn\\
	\ophat{Z}_t^\mu
	\doteq -\op{M} - \op{M}\, (\op{Z}_t^\mu - \op{M})^{-1} \, \op{M},
	\nn
\end{gather}
and the inner product $\langle \cdot\,, \cdot \rangle_\sharp$ is as per the theorem statement.
Recalling \er{eq:op-M}, \er{eq:eig-XYZ}, these operators are necessarily also of the spectral form \er{eq:spectral-opA}, with their respective eigenvalues given by inspection by
\begin{gather}
	[\hat x_t^\mu]_n
	\doteq [x_t^\mu]_n - \frac{[y_t^\mu]_n^2}{[z_t^\mu]_n - m_n},
	\quad
	[\hat y_t^\mu]_n
	\doteq - \frac{[y_t^\mu]_n\, m_n}{[z_t^\mu]_n - m_n},
	\quad
	\nn\\
	[\hat z_t^\mu]_n
	\doteq -m_n - \frac{m_n^2}{[z_t^\mu]_n - m_n},
	\nn
\end{gather}
for all $n\in\N$. After applying \er{eq:eig-XYZ}, \er{eq:eig-XYZ-alt}, sum-of-angle manipulations yield
\begin{align}
	& 
	[\hat x_t^\mu]_n = [p_t^\mu]_n,
	\quad
	[\hat y_t^\mu]_n = [q_t^\mu]_n,
	\quad
 	[\hat z_t^\mu]_n = [p_t^\mu]_n,
	\label{eq:eig-equal}
\end{align}
for all $n\in\N$, where $[p_t^\mu]_n$, $[q_t^\mu]_n$ are as per \er{eq:eig-pqr}. For example, for the second equality,
\begin{align}
	[\hat y_t^\mu]_n
	& 
	\doteq - \frac{[y_t^\mu]_n\, m_n}{[z_t^\mu]_n - m_n}
	= - \frac{m_n\, \cos\theta_n^\mu\, \csc(\omega_n^\mu\, t + \theta_n^\mu)}
				{\cos^2\theta_n^\mu \, [ \, \cot\theta_n^\mu - \cot(\omega_n^\mu\, t + \theta_n^\mu) \, ]}
	\nn\\
	& = \frac{1}{\omega_n^\mu} \frac{\csc(\omega_n^\mu\, t + \theta_n^\mu)}
			{(\frac{-1}{\omega_n^\mu\, m_n})\, \cos\theta_n^\mu \, [ \, \cot\theta_n^\mu - \cot(\omega_n^\mu\, t + \theta_n^\mu) \, ]}
	\nn\\
	& = \frac{1}{\omega_n^\mu} \frac{1}{\sin(\omega_n^\mu\, t + \theta_n^\mu)\, \cos\theta_n^\mu 
					- \cos(\omega_n^\mu\, t + \theta_n^\mu)\, \sin\theta_n^\mu\, }
	\nn\\
	& = \frac{1}{\omega_n^\mu\, \sin(\omega_n^\mu\,t)} = [q_t^\mu]_n
	\nn
\end{align}
for all $n\in\N$. The other two equalities in \er{eq:eig-equal} follow similarly.
Consequently, $\ophat{X}_t^\mu = \op{P}_t^\mu = \ophat{Z}_t^\mu$ and $\ophat{Y}_t^\mu = \op{Q}_t^\mu = (\op{Q}_t^\mu)'$ by \er{eq:op-PQR}, \er{eq:eig-equal}, so that \er{eq:G-mu} follows by \er{eq:pre-kernel-G}.
\end{proof}

\begin{remark}
The Hessian operator in \er{eq:G-mu} may also be interpreted as the solution of a differential Riccati equation \cite{DM1:15} that arises in an optimal control problem of the form \er{eq:value-mu} with $\psi\doteq\delta^-(\cdot,\zeta)$, i.e. the optimal TPBVP \er{eq:G-mu-def}, where $\delta^-$ is as per \er{eq:delta-minus}.
\hfill{$\square$}
\end{remark}
Some useful properties of the operators $\op{E}_\mu$ and $\op{P}_t^\mu$, $\op{Q}_t^\mu$ of \er{eq:op-E} and \er{eq:op-PQR} follow by generalising results from \cite[Appendix B]{DM1:17}. These properties find application in the group constructions to follow.

\begin{lemma}
\label{lem:op-properties}
Given $\mu\in\R_{>0}$, $t\in(0,\bar t^\mu)$, operators $\op{E}_\mu$, $\op{P}_t^\mu$, $\op{Q}_t^\mu$ of \er{eq:op-E}, \er{eq:op-PQR} are bounded and boundedly invertible, with 
$$ 
	\op{E}_\mu\in\bo(\cX_1;\cX), \ \op{E}_\mu^{-1}\in\bo(\cX;\cX_1), \quad
	\op{P}_t^\mu, \op{Q}_t^\mu, (\op{P}_t^\mu)^{-1}, (\op{Q}_t^\mu)^{-1}\in\bo(\cX_1).
$$
Moreover, given $\omega_n^\mu$ as per \er{eq:omega-lambda},
\begin{gather}
	\op{E}_\mu\, \xi = \sum_{n=1}^\infty 
	\rev{\omega_n^\mu\, }
	\langle \xi,\, \tilde\varphi_n \rangle_1\, \varphi_n,
	\quad
	(\op{P}_t^\mu)^{-1}\, \xi 
	= -\sum_{n=1}^\infty \omega_n^\mu\, \tan(\omega_n^\mu\,t )\, \langle \xi, \, \tilde\varphi_n \rangle_1 \, \tilde\varphi_n\,,
	\nn\\
	\op{E}_\mu^{-1}\, \pi
	= \sum_{n=1}^\infty 
	\rev{\ts{\frac{1}{\omega_n^\mu}}\,}
	\langle \pi,\, \varphi_n \rangle\, \tilde\varphi_n,
	\quad
	(\op{Q}_t^\mu)^{-1}\, \xi
	= \sum_{n=1}^\infty \omega_n^\mu\, \sin(\omega_n^\mu\, t)\, \langle \xi, \, \tilde\varphi_n \rangle_1 \, \tilde\varphi_n\,,
	\nn\\
	-(\op{Q}_t^\mu)^{-1}\, \op{P}_t^\mu\, \xi
	=  \sum_{n=1}^\infty \cos(\omega_n^\mu\, t)\, \langle \xi, \, \tilde\varphi_n \rangle_1 \, \tilde\varphi_n\,,
	\nn\\
	(\op{Q}_t^\mu)^{-1}\, (\op{E}_t^\mu)^{-1}\, \pi
	= \sum_{n=1}^\infty \sin(\omega_n^\mu\,t) \, \langle \pi, \, \varphi_n \rangle \, \tilde\varphi_n\,,
	\nn\\
	-\op{E}_\mu\, \op{Q}_t^\mu \, (\op{I} - [(\op{Q}_t^\mu)^{-1}\, \op{P}_t^\mu]^2 ) \, \xi
	= -\sum_{n=1}^\infty \sin(\omega_n^\mu\, t) \, \langle \xi, \, \tilde\varphi_n \rangle_1\, \varphi_n\,,
	\nn\\
	-\op{E}_\mu\, \op{P}_t^\mu\, (\op{Q}_t^\mu)^{-1}\, \op{E}_\mu^{-1}
	= \sum_{n=1}^\infty \cos(\omega_n^\mu\,t) \, \langle \pi,\, \varphi_n \rangle \, \varphi_n\,,
	\label{eq:op-properties}
\end{gather}
for all $\xi\in\cX_1$, $\pi\in\cX$.
\end{lemma}
\begin{proof}
The first and last equalities in \er{eq:op-properties} and associated boundedness properties are demonstrated below. The remaining equalities and bounds follow using analogous arguments.

{\em First equality in \er{eq:op-properties}:} Fix $\mu\in\R_{>0}$, $t\in(0,\bar t^\mu)$, $\xi\in\cX_1$, $\pi\in\cX$. Recall by \er{eq:op-I-mu}, \er{eq:op-E}, \er{eq:spectral-opA}, \er{eq:omega-lambda} that $\Lambda^\half:\cX_1\rightarrow\cX$, $\op{I}_\mu^\half:\cX\rightarrow\cX_1$, and $\op{I}_\mu^\half\, \Lambda^\half:\cX_1\rightarrow\cX_1$ satisfy
\begin{align}
	\Lambda^\half\, \xi
	& = \sum_{n=1}^\infty \sqrt{\lambda_n} \, \langle \xi,\, \tilde\varphi_n \rangle_1\, \tilde\varphi_n
	=  \sum_{n=1}^\infty \langle \xi,\, \tilde\varphi_n \rangle_1\, \varphi_n\,,
	\nn\\
	\op{I}_\mu^\half\, \pi
	& = \sum_{n=1}^\infty \frac{1}{\sqrt{1 + \mu\, \lambda_n}}\, \langle \pi,\, \varphi_n \rangle\, \varphi_n
	= \sum_{n=1}^\infty \omega_n^\mu\, \langle \pi,\, \varphi_n \rangle\, \tilde\varphi_n\,,
	\nn\\
	\op{I}_\mu^\half\, \Lambda^\half\, \xi
	& = \sum_{n=1}^\infty \omega_n^\mu\, \langle \Lambda^\half\, \xi,\, \varphi_n \rangle \, \tilde\varphi_n
	= \sum_{n=1}^\infty \omega_n^\mu\, \left\langle \sum_{k=1}^\infty \langle \xi,\, \tilde\varphi_k \rangle_1\, \varphi_k, \, 
						\varphi_n \right\rangle \, \tilde\varphi_n
	\nn\\
	& \hspace{30mm}
	= \sum_{n=1}^\infty \omega_n^\mu\, \langle \xi,\, \tilde\varphi_n \rangle_1\, \tilde\varphi_n\,.
	\nn
\end{align}
Hence, $\op{E}_\mu = \Lambda^\half\,\op{I}_\mu^\half\, \Lambda^\half:\cX_1\rightarrow\cX$ satisfies
\begin{align}
	\op{E}_\mu\, \xi
	& = \Lambda^\half\, (\op{I}_\mu^\half\, \Lambda^\half)\, \xi 
	= \sum_{n=1}^\infty \langle \op{I}_\mu^\half\, \Lambda^\half\, \xi,\, 
						\tilde\varphi_n \rangle_1\, \varphi_n
	\nn\\
	& 
	= \sum_{n=1}^\infty \left\langle \sum_{k=1}^\infty \omega_k^\mu\, \langle \xi,\, \tilde\varphi_k \rangle_1\, \tilde\varphi_k,\, 
						\tilde\varphi_n \right\rangle_1 \varphi_n
	= \sum_{n=1}^\infty \omega_n^\mu\, \langle \xi,\, \tilde\varphi_n \rangle_1\, \varphi_n\,,
	\label{eq:op-E-explicit}
\end{align}
as per the first equality in \er{eq:op-properties}. Note by inspection that $\|\op{E}_\mu\|_{\bo(\cX_1;\cX)} \le \sup_{n\in\N} |\omega_n^\mu| = \frac{1}{\mu} < \infty$.

{\em Last equality in \er{eq:op-properties}:} 
Fix $\mu\in\R_{>0}$, $t\in(0,\bar t^\mu)$. By inspection of \er{eq:omega-lambda}, \er{eq:op-PQR}, \er{eq:eig-pqr}, note that
$|([q_t^\mu]_n)^{-1}| = |\omega_n^\mu| \, |\sin(\omega_n^\mu\, t)| \le \ts{\frac{1}{\mu}}$ for all $n\in\N$. Consequently, a bounded operator of the form \er{eq:spectral-opA} is defined by
\begin{align}
	\op{R}_t^\mu\, \xi
	\doteq \sum_{n=1}^\infty ([q_t^\mu]_n)^{-1} \, \langle \xi, \, \tilde\varphi_n \rangle_1 \, \tilde\varphi_n
	\nn
\end{align}
for all $\xi\in\cX_1$, with $\|\op{R}_t^\mu\|_{\bo(\cX_1)} \le \frac{1}{\mu}$. Fix any $\xi\in\cX_1$, $\pi\in\cX$. Recalling \er{eq:op-PQR},
\begin{align}
	\op{R}_t^\mu\, \op{Q}_t^\mu\, \xi
	& = \sum_{n=1}^\infty ([q_t^\mu]_n)^{-1} \left\langle \sum_{k=1}^\infty [q_t^\mu]_k \, \langle\xi,\, \tilde\varphi_k \rangle_1 \, 
		\tilde\varphi_k , \, \tilde\varphi_n \right\rangle_1 \tilde\varphi_n
	\nn\\
	&  = \sum_{n=1}^\infty \sum_{k=1}^\infty \frac{[q_t^\mu]_k}{[q_t^\mu]_n} \langle\xi,\, \tilde\varphi_k\rangle_1\, 
	\langle \tilde\varphi_k,\, \tilde\varphi_n \rangle_1\, \tilde\varphi_n
	= \sum_{n=1}^\infty \langle\xi,\, \tilde\varphi_n\rangle_1\, \tilde\varphi_n = \xi,
	\nn
\end{align}
so that $(\optilde{Q}_t^\mu)^{-1} \doteq \optilde{R}_t^\mu \in \bo(\cX_1)$. Similarly, recalling \er{eq:op-E-explicit}, 
\begin{align}
	\op{E}_\mu^{-1}\, \pi
	& = \sum_{n=1}^\infty \frac{1}{\omega_n^\mu} \, \langle \pi, \, \varphi_n \rangle\, \tilde\varphi_n\,,
	\label{eq:op-E-inv-explicit}
\end{align}
and note further that $\|\op{E}_\mu^{-1}\|_{\bo(\cX;\cX_1)} \le \sup_{n\in\N} 1/|\omega_n^\mu| = 1/|\omega_1^\mu| < \infty$. Applying \er{eq:op-PQR}, \er{eq:eig-pqr}, \er{eq:op-E-explicit}, \er{eq:op-E-inv-explicit}, and the definition of $\op{R}_t^\mu = (\op{Q}_t^\mu)^{-1}$ above, analogous calculations yield
\begin{align}
	-\op{E}_\mu\, \op{P}_t^\mu\, (\op{Q}_t^\mu)^{-1}\, \op{E}_\mu^{-1}\, \pi
	& = - \sum_{n=1}^\infty \omega_n^\mu\, \frac{[p_t^\mu]_n}{[q_t^\mu]_n} \, \frac{1}{\omega_n^\mu} \, \langle \pi, \, \varphi_n \rangle\, \varphi_n
	\nn\\
	& 
	= \sum_{n=1}^\infty \frac{\sin(\omega_n^\mu\, t)}{\tan(\omega_n^\mu\, t)} \, \langle \pi, \, \varphi_n \rangle\, \varphi_n
	= \sum_{n=1}^\infty \cos(\omega_n^\mu\, t) \, \langle \pi, \, \varphi_n \rangle\, \varphi_n\,,
	\nn
\end{align}
as per the last equality in \er{eq:op-properties}. By inspection, $\|-\op{E}_\mu\, \op{P}_t^\mu\, (\op{Q}_t^\mu)^{-1}\, \op{E}_\mu^{-1}\|_{\bo(\cX)} \le 1$.
\end{proof}

\begin{remark}
\label{rem:escape}
By inspection of \er{eq:op-PQR}, \er{eq:eig-pqr}, along with Lemma \ref{lem:op-properties}, the respective eigenvalues of operators $\op{P}_t^\mu$, $(\op{P}_t^\mu)^{-1}$, $\op{Q}_t^\mu\in\bo(\cX_1)$ experience finite escape behaviour, with
\begin{align}
	& \lim_{t\rightarrow(\ts{\frac{j\, \pi}{\omega_n^\mu}})} |[p_t^\mu]_n| = \infty =
	\lim_{t\rightarrow(\ts{\frac{j\, \pi}{\omega_n^\mu}})} |[q_t^\mu]_n|\,,
	\quad
	\lim_{t\rightarrow(\ts{\frac{(j - \frac{1}{2})\, \pi}{\omega_n^\mu}})} \ts{\frac{1}{|[p_t^\mu]_n|}} = \infty,
	\quad n,j\in\N.
	\label{eq:escape}
\end{align}
The first of these escape times is $\inf_{n\in\N} \ts{\frac{\pi}{2\, \omega_n^\mu}} = (\ts{\frac{\pi}{2\sqrt{2}}})\, \bar t^\mu \approx 1.11\, \bar t^\mu, $ which is accompanied by an anticipated loss of concavity of $J_t^\mu(\xi,\cdot)$ for horizons beyond $\bar t^\mu$, for any $\xi\in\cX_1$.
\hfill{$\square$}
\end{remark}


\section{Group construction \rev{via optimal control}}
\label{sec:group}

Hamilton's action principle suggests that the characteristic system associated with the optimal control problem \er{eq:value-mu} may be used to represent all solution of the wave equation \er{eq:wave} via its approximation \er{eq:wave-mu}. This motivates construction of a prototype fundamental solution semigroup for \er{eq:wave-mu}, and its subsequent validation via Theorem \ref{thm:generators}. The finite escape behaviour identified in Remark \ref{rem:escape} further suggests that this construction proceed for arbitrary horizons via the temporal concatenation of sufficiently many sufficiently short horizons, using the aforementioned short horizon prototype.


\subsection{Short horizon prototype}
A prototype \rev{element of} the group $\{\cU_s^\mu\}_{s\in\R}$ of \er{eq:group-mu} may be constructed \cite{DM3:15,DM4:15} \rev{on a short horizon} via a special case of the optimal control problem \er{eq:value-mu}, using the idempotent representation \er{eq:idem-rep}, \er{eq:G-mu}. In particular, \rev{a fixed short horizon $t\in(0,\bar t^\mu)$ and specific terminal payoff $\psi = \psi_v:\cX_1\rightarrow\R$ are} considered in \er{eq:value-mu}, with
\begin{align}
	\psi(\xi) = \psi_v(\xi) & \doteq \langle \xi\,, \op{E}_\mu^{-1}\, v  \rangle_1,
	\quad \xi\in\cX_1,
	\label{eq:psi-v}
\end{align}
for any fixed $v\in\cX$, in which $\op{E}_\mu^{-1}\in\bo(\cX;\cX_1)$ by \rev{Lemma \ref{lem:op-properties}.}
As $G_t^\mu(\cdot,\zeta)$, $G_t^\mu(\xi,\cdot)$, and $\rev{\psi =\, } \psi_v$ in \er{eq:idem-rep} are {\Frechet} differentiable for any $\xi,\zeta\in\cX_1$, the supremum there must be achieved where the Riesz representation of the {\Frechet} derivative of $G_t^\mu(\xi,\cdot) + \psi_v(\cdot)$ is zero. That is, 
\begin{align}
	0 & = \grad_\zeta [G_t^\mu(\xi,\zeta) + \psi_v(\zeta) ]_{\zeta=\zeta_\xi^*}
	= \op{Q}_t^\mu\, \xi + \op{P}_t^\mu\, \zeta_\xi^* + \op{E}_\mu^{-1}\, v
	\label{eq:idem-rep-Frechet}
\end{align}
for any $\zeta_\xi^*\in\argmax_{\zeta\in\cX_1} \{ G_t^\mu(\xi,\zeta) + \psi_v(\zeta) \}$. As $\op{P}_t^\mu$ is \rev{boundedly invertible for $t\in(0,\bar t^\mu)$ by Lemma \ref{lem:op-properties},} the achieved terminal state $\zeta_\xi^*$ is defined uniquely by \er{eq:idem-rep-Frechet}, and representation \er{eq:idem-rep} of $W_t^\mu(x)$ subsequently follows. In particular,
\begin{equation}
	\begin{aligned}
	& \zeta_\xi^*
	= -(\op{P}_t^\mu)^{-1} \left( \op{Q}_t^\mu\, \xi + \op{E}_\mu^{-1}\, v \right),
	&&& 
	W_t^\mu(\xi) 
	= G_t^\mu(\xi,\zeta_\xi^*) + \psi_v(\zeta_\xi^*)\,,
	\end{aligned}
	\label{eq:z-star}
\end{equation}
for all $\xi\in\cX_1$. With a view to computing the corresponding optimal input \er{eq:w-star}, note via \er{eq:idem-rep}, \er{eq:idem-rep-Frechet} and the chain rule that
\begin{align}
	& \grad W_t^\mu(\xi)
	= \grad [ G_t^\mu(\xi,\zeta_\xi^*) + \psi_v(\zeta_\xi^*)]
	\nn\\
	& = \grad_\xi G_t^\mu(\xi,\zeta)|_{\zeta=\zeta_\xi^*} + (D_\xi \zeta_\xi^*)\, \grad_\zeta [ G_t^\mu(\xi,\zeta) + \psi_v(\zeta)]_{\zeta = \zeta_\xi^*}
	= \grad_\xi G_t^\mu(\xi,\zeta)|_{\zeta=\zeta_\xi^*},
	\nn
\end{align}
in which $D_\xi \zeta_\xi^*\in\bo(\cX_1)$ is the Frechet derivative of the mapping $\xi\mapsto \zeta_\xi^*$, $\xi\in\cX_1$, and the final equality follows by \er{eq:idem-rep-Frechet}. Hence, recalling \rev{\er{eq:idem-rep},} \er{eq:G-mu}, \er{eq:z-star},
\begin{align}
	& \grad W_t^\mu(\xi)
	= \op{P}_t^\mu \xi + \op{Q}_t^\mu \zeta_\xi^*
	= ( \op{P}_t^\mu -\op{Q}_t^\mu\, (\op{P}_t^\mu)^{-1}\, \op{Q}_t^\mu) \, \xi
	- \op{Q}_t^\mu\, (\op{P}_t^\mu)^{-1}\, \op{E}_\mu^{-1}\, v.
	\label{eq:grad-W-mu}
\end{align}
From \er{eq:w-star}, the optimal control is
\begin{align}
	w_s^*
	& = \op{I}_\mu^{\half}\, \pi_s,
	\quad \pi_s\doteq\op{E}_\mu\, \grad W_{t-s}^\mu(\xi_s),
	\quad s\in[0,t].
	\nn
\end{align}
The specific choice \er{eq:psi-v} of terminal payoff $\psi = \psi_v$ yields
\begin{align}
	& \pi_t
	= \op{E}_\mu\, \grad W_0(\xi_t) = \op{E}_\mu\,\grad \psi_v(\xi_t) =   \op{E}_\mu\,  \op{E}_\mu^{-1}\, v = v.
	\nn
\end{align}
Hence, \er{eq:idem-rep}, \er{eq:z-star}\rev{, \er{eq:grad-W-mu}} further imply that 
\begin{equation}
	\begin{aligned}
	& \xi_t 
	= \zeta_{\xi_0}^*
	= - (\op{P}_t^\mu)^{-1} \, \op{Q}_t^\mu\, \xi_0 - (\op{P}_t^\mu)^{-1} \, \op{E}_\mu^{-1}\, \pi_t,
	\\
	& \pi_0
	=  \op{E}_\mu\, \grad W_t(\xi_0)
	= \op{E}_\mu\, ( \op{P}_t^\mu -\op{Q}_t^\mu\, (\op{P}_t^\mu)^{-1}\, \op{Q}_t^\mu) \, \xi_0
	-  \op{E}_\mu\, \op{Q}_t^\mu\, (\op{P}_t^\mu)^{-1}\, \op{E}_\mu^{-1}\, \pi_t.
	\end{aligned}
	\label{eq:pi-0}
\end{equation}
By exploiting invertibility of the operators involved, \rev{see Lemma \ref{lem:op-properties},} some straightforward manipulations yield that
\begin{align}
	& \hspace{-3mm}
	\left( \ba{c}
		\xi_t
		\\
		\pi_t
	\ea \right)
	= \wh{\op{U}}_t^\mu
		\left( \ba{c}
		\xi_0
		\\
		\pi_0
	\ea \right)\!, \quad
	\wh{\op{U}}_t^\mu
	\doteq
	 \left( \ba{c|c}
	 	[\wh{\op{U}}_t^\mu]_{11} & [\wh{\op{U}}_t^\mu]_{12}
		\\[0mm]
		& \\[-3mm]\hline
		& \\[-3mm]
	 	[\wh{\op{U}}_t^\mu]_{21} & [\wh{\op{U}}_t^\mu]_{22}
	\ea \right)\!,
	\label{eq:prototype-mu}
\end{align}
where $[\wh{\op{U}}_t^\mu]_{11}\in\bo(\cX_1)$, $[\wh{\op{U}}_t^\mu]_{12}\in\bo(\cX;\cX_1)$, $[\wh{\op{U}}_t^\mu]_{21}\in\bo(\cX_1;\cX)$, $[\wh{\op{U}}_t^\mu]_{22}\in\bo(\cX)$, $t\in(0,\bar t^\mu)$, are given by
\begin{equation}
	\begin{aligned}
	{[\wh{\op{U}}_t^\mu]_{11}}
	& \doteq 
	-(\op{Q}_t^\mu)^{-1}\, \op{P}_t^\mu,
	&
	[\wh{\op{U}}_t^\mu]_{12}
	& \doteq (\op{Q}_t^\mu)^{-1}\, \op{E}_\mu^{-1},
	\\
	[\wh{\op{U}}_t^\mu]_{21}
	& \doteq \rev{-} \op{E}_\mu\, \op{Q}_t^\mu\, \left( \op{I} - [(\op{Q}_t^\mu)^{-1} \, \op{P}_t^\mu]^2 \right),
	&
	[\wh{\op{U}}_t^\mu]_{22}
	& \doteq - \op{E}_\mu\, \op{P}_t^\mu\, (\op{Q}_t^\mu)^{-1}\, \op{E}_\mu^{-1}.
	\end{aligned}
	\label{eq:prototype-mu-elmts}
\end{equation}
\rev{
\begin{lemma}
\label{lem:short-horizon-prototype}
Given $\mu\in\R_{>0}$, $t\in(0,\bar t^\mu)$, the operators $\op{U}_t^\mu, \wh{\op{U}}_t^\mu\in\bo(\cY)$ of \er{eq:group-mu}, \er{eq:group-mu-elmts} and \er{eq:prototype-mu}, \er{eq:prototype-mu-elmts} are equivalent, i.e.
$\wh{\op{U}}_t^\mu = \op{U}_t^\mu$.
\end{lemma}
\begin{proof}
Fix $\mu\in\R_{>0}$, $t\in(0,\bar t^\mu)$, $\xi\in\cX_1$, $\pi\in\cX$. The assertion follows by comparing \er{eq:group-mu}, \er{eq:group-mu-elmts} with \er{eq:prototype-mu}, \er{eq:prototype-mu-elmts}, via Lemma \ref{lem:op-properties}.
\end{proof}
}

\if{false}

\question{XX exclude the following associative stuff for now, leave it until the long horizon? XX}

\rev{$\wh{\op{U}}_t^\mu = \op{U}_t^\mu$ for all $t\in(0,\bar t^\mu)$.
Moreover, as the eigenvalues $[p_t^\mu]_n$, $[q_t^\mu]_n$ of $\op{P}_t^\mu$, $\op{Q}_t^\mu$ are odd functions of $t$, the bounded linear operator components of $\wh{\op{U}}_t^\mu$ appearing in \er{eq:prototype-mu-elmts} may also be defined for $t\in(-\bar t^\mu,0)$. Meanwhile, with $t=0$ and $\xi_0,\zeta\in\cX_1$, \er{eq:idem-rep} implies that $G_t^\mu(\xi_0,\zeta) = \delta_{\zeta}^-(\xi_0) $, so that $\xi_t = \zeta_{\xi_0}^* = \xi_0$ and $\pi_t = \pi_0$. That is, $\wh{\op{U}}_t^\mu$ is also defined at $t=0$, with $\wh{\op{U}}_0^\mu = \op{I}$. Some straightforward manipulations then yield that $\wh{\op{U}}_{-t}^\mu\ \wh{\op{U}}_t^\mu = \op{I} = \wh{\op{U}}_{t}^\mu\ \wh{\op{U}}_{-t}^\mu$ for all $t\in(-\bar t^\mu,\bar t^\mu)$, and more generally \cite{DM4:15} 
\begin{align}
	\wh{\op{U}}_s^\mu\ \wh{\op{U}}_t^\mu
	= \wh{\op{U}}_{s+t}^\mu, 
	\quad
	\wh{\op{U}}_t^\mu
	& = \op{U}_t^\mu,
	\quad s,t,s+t\in(-\bar t^\mu,\bar t^\mu).
	\label{eq:group-mu-elmt-prototype}
\end{align}
Consequently, the set $\{\wh{\op{U}}_t^\mu\}_{t\in(-\bar t^\mu,\bar t^\mu)}$ 
can} be regarded as a {\em prototype} for the group $\{ \op{U}_t^\mu \}_{t\in\R}$ defined by \er{eq:group-mu}, \er{eq:group-mu-elmts}. \rev{As this prototype is defined entirely via the optimal control problem \er{eq:value-mu}, \er{eq:psi-v}, it is emphasised that the associative group property involved is restricted to holding on $(-\bar t^\mu,\bar t^\mu)$. Its extension to longer horizon requires extra work.}

\question{XX Formalise a proof of \er{eq:group-mu-elmt-prototype}, etc? Would be as per \cite{DM4:15}. XX}

\fi


\subsection{Longer horizons}
The correspondence between stationary action and optimal control can break down for longer time horizons due to a loss of concavity of the payoff \er{eq:payoff}, see Lemma \ref{lem:concave} and the finite escape property \er{eq:escape} associated with the  idempotent representation \er{eq:idem-rep}, \er{eq:G-mu}, \er{eq:op-PQR}, \er{eq:eig-pqr}. Consequently, for longer horizons, a modified approach is required. The basis for two such approaches has been proposed for finite dimensional problems, see \cite{MD1:17,MD1:15}, based on replacing the {\em sup} operation in \er{eq:value-mu}, \er{eq:idem-rep} with a {\em stat} operation. In particular, this {\em stat} operation can be used to define a value function analogous to \er{eq:value-mu}, \er{eq:idem-rep} corresponding to a \rev{\em stationary} payoff, without the need to assume that stationarity is achieved at a \rev{maximum}. Alternatively, by retaining the sup operation in \er{eq:value-mu} on shorter horizons, longer time horizons can be accumulated by concatenating these short horizons, with the stat operation used to relax the constraints associated with the intermediate states joining adjacent horizons via a generalisation of $G_t^\mu$ in \er{eq:idem-rep}, \er{eq:G-mu}.
Here, this latter approach is considered, with 
an appropriate definition of the {\em stat} operation given by
\begin{align}
	& \stat_{\zeta\in\cX_1} F(\zeta)
	\doteq \left\{ F(\bar\zeta)\, \biggl|\, \bar\zeta\in\argstat_{\zeta\in\cX_1} F(\zeta) \right\},
	\ F:\cX_1\rightarrow\R,
	\nn
	\\
	& \argstat_{\zeta\in\cX_1} F(\zeta)
	\doteq \!\left\{ \zeta\in\cX_1 \, \biggl| \, 0 = \lim_{y\rightarrow\zeta} \! \frac{|F(y) - F(\zeta)|}{\|y - \zeta\|_1} \! \right\}\!.
	\label{eq:stat}
\end{align}
%
With a view to formalising the aforementioned concatenation approach, with $\mu\in\R_{>0}$ and $x\in\cX_1$ arbitrary and fixed, consider any longer horizon $t\in[\bar t^\mu,\infty)$ of interest for which the payoff $J_t^\mu(x,\cdot)$ of \er{eq:payoff-mu} is not concave. The key idea is to select a sufficiently large number $n_t\in\N$ of shorter horizons $\tau\doteq t/n_t\in(0,\bar t^\mu)$ such that $J_\tau^\mu(x,\cdot)$ is concave by Lemma \ref{lem:concave}. Consequently, concavity of $J_\tau^\mu(\zeta_k,\cdot)$ is retained on every subinterval $[(k-1)\,\tau, \, k\,\tau]$, $k\in [1,n_t]\cap\N$, where $\zeta_k = \xi_{k\tau}\in\cX_1$ denotes the state at the corresponding intermediate time.

In further formalising this approach, it is useful to propose a candidate generalisation of the value function $W_t^\mu$ of \er{eq:value-mu} via a corresponding generalisation of \er{eq:idem-rep}. To this end, given $\mu\in[0,1]$, define a set of longer horizons $\Omega^\mu\subset\R_{>0}$ by
\begin{align}
	& \Omega^\mu
	\doteq \R_{>0} \setminus \{ t_{n,j,k}^\mu \}_{n\in\N\cup\{\infty\}, j,k\in\N},
	\label{eq:Omega-mu}
	\\
	\quad
	& t_{n,j,k}^\mu
	\doteq \ts{\frac{j}{k}}\, \ts{( \frac{\pi}{2}) \, \frac{1}{\omega_n^\mu}}, 
	\quad \ts{\frac{1}{\omega_n^\mu}} = (\ts{\frac{1}{\lambda_n^\mu}})^\half = (\ts{\frac{1}{\lambda_n}} + \mu^2)^\half,
	\nn\\
	& t_{\infty,j,k}^\mu
	\doteq \ts{\frac{j}{k}}\, (\ts{\frac{\pi}{2}})\, \mu.
	\nn
\end{align}
%
\begin{lemma}
\label{lem:Omega-mu-properties}
$\Omega^\mu$ of \er{eq:Omega-mu} satisfies the following properties:
\begin{enumerate}[\em (i)]
\item $\Omega^\mu$ is uncountable and dense in $\R_{\ge 0}$, and has the same measure; and
\item $\Omega^\mu$ is closed under addition, and scaling by elements of $\Q_{>0}$, with
\begin{align}
	\begin{aligned}
	s,t\in\Omega^\mu
	& \ \Longrightarrow\ s+t\in\Omega^\mu\,,
	\\
	\gamma\in\Q_{>0},\,
	t\in\Omega^\mu
	& \ \Longrightarrow\ \gamma\, t\in\Omega^\mu\,. 
	\end{aligned}
	\label{eq:rational-closure}
\end{align}
\end{enumerate}
\vspace{-1mm}
\end{lemma}
\begin{proof}
As $\{ t_{n,j,k}^\mu \}_{n,j,k\in\N}$ is countable, 
assertion {\em (i)} is immediate.
For assertion {\em (ii)}, note by \er{eq:Omega-mu} that 
\begin{align}
	s,t\in\Omega^\mu
	& \ \Longleftrightarrow \
	s \ne \ts{\frac{j}{k}}\, \ts{( \frac{\pi}{2}) \, \frac{1}{\omega_n^\mu}}\,,
	\
	t \ne \ts{\frac{\hat j}{\hat k}}\, \ts{( \frac{\pi}{2}) \, \frac{1}{\omega_{\hat n}^\mu}}\,,
	\nn
\end{align}
for all $j,k,\hat j,\hat k\in\N$, $n,\hat n\in\N\cup\{\infty\}$. Fix any $p,q\in\N$, $r\in\N\cup\{\infty\}$. Selecting \rev{in particular} $j = \hat j = p$, $k = \hat k = 2\, q$, $n = \hat n = r$, 
\begin{align}
	s + t
	& \ne \ts{\frac{p}{2\, q}}\, \ts{( \frac{\pi}{2}) \, \frac{1}{\omega_r^\mu}} + \ts{\frac{p}{2\, q}}\, \ts{( \frac{\pi}{2}) \, \frac{1}{\omega_r^\mu}}
	= \ts{\frac{p}{q}}\, \ts{( \frac{\pi}{2}) \, \frac{1}{\omega_r^\mu}}\,.
	\nn
\end{align}
As $p,q\in\N$, $r\in\N\cup\{\infty\}$ are arbitrary, $s+t\in\Omega^\mu$ by \er{eq:Omega-mu}.
Similarly, for any $p,q\in\N$ fixed,
\begin{align}
	t\in\Omega^\mu
	& \quad \Longleftrightarrow \quad
	\ts{\frac{p}{q}}\, t \ne \ts{\frac{p\, j}{q\, k}}\, \ts{( \frac{\pi}{2}) \, \frac{1}{\omega_n^\mu}} 
	\quad \forall \ n,j,k\in\N.
	\nn
\end{align}
Selecting $j \doteq q\, \hat j$, $k \doteq p\, \hat k$, for any $\hat j, \hat k\in\N$, yields
\begin{align}
	\ts{\frac{p}{q}}\, t
	& \ne \ts{\frac{p\, q\, \hat j}{q\, p\, \hat k}}\, \ts{( \frac{\pi}{2}) \, \frac{1}{\omega_n^\mu}} 
	= \ts{\frac{\hat j}{\hat k}}\, \ts{( \frac{\pi}{2}) \, \frac{1}{\omega_n^\mu}} 
	\quad \forall \ n,\hat j, \hat k\in\N.
	\nn
\end{align}
As $p,q\in\N$ are arbitrary, defining $\gamma \doteq \frac{p}{q}$ yields \er{eq:rational-closure}.
\end{proof}

In view of definition \er{eq:G-mu} of $G_t^\mu$ for $t\in(0,\bar t^\mu)$, and \er{eq:Omega-mu}, define $\wt{G}_t^\mu:\cX_1\times\cX_1\rightarrow\R$ by
\begin{align}
	\wt{G}_t^\mu(\xi,\zeta)
	& \doteq \demi \left\langle \left( \ba{c}
				\xi \\ \zeta
			\ea \right)\!, \left( \ba{cc}
				\optilde{P}_t^\mu & \optilde{Q}_t^\mu
				\\
				\optilde{Q}_t^\mu & \optilde{P}_t^\mu
			\ea \right)  \! \left( \ba{c}
				\xi \\ \zeta
			\ea \right)
		\right\rangle_\sharp\!\!,
	\label{eq:G-mu-long}
\end{align}
for all $t\in\Omega^\mu$, in which $\optilde{P}_t^\mu$, $\optilde{Q}_t^\mu$ are defined analogously to \er{eq:op-PQR}, with
\begin{equation}
	\begin{aligned}
	\optilde{P}_t^\mu \, \xi
	& \doteq \sum_{n=1}^\infty [\tilde p_t^\mu]_n\, \langle \xi,\tilde\varphi_n \rangle_1\, \tilde\varphi_n, 
	&
	\optilde{Q}_t^\mu \, \xi
	& \doteq \sum_{n=1}^\infty [\tilde q_t^\mu]_n\, \langle \xi,\tilde\varphi_n \rangle_1\, \tilde\varphi_n, 
	\quad t\in\Omega^\mu, \ \xi\in\cX_1,
	\end{aligned}
	\label{eq:optilde-PQR-mu}
\end{equation}
in which the respective eigenvalues are given by
\begin{align}
	& [\tilde p_t^\mu]_n
	\doteq \frac{-1}{\omega_n^\mu\, \tan(\omega_n^\mu\, t)},
	\quad
	[ \tilde q_t^\mu]_n
	\doteq \frac{1}{\omega_n^\mu\, \sin(\omega_n^\mu\, t)},
	\label{eq:eig-pqr-tilde-mu}
\end{align}
for all $n\in\N$, $t\in\Omega^\mu$.
\rev{Given \er{eq:G-mu}, \er{eq:G-mu-long}, note that for $\mu\in(0,1]$, $t\in(0,\bar t^\mu)\cap\Omega^\mu$,
\begin{align}
	\wt{G}_t^\mu(\xi,\zeta) 
	& = G_t^\mu(\xi,\zeta)
	\quad \forall  \ \xi,\zeta\in\cX_1.
	\label{eq:G-mu-long-and-short}
\end{align}
}

\begin{lemma}
\label{lem:optilde-PQR-bounded-mu}
Given any $\mu\in[0,1]$, $t\in\Omega^\mu$, there exists an $L_t\in\R_{>0}$ independent of $\mu$ such that
\begin{align}
	& \max(|[\tilde p_t^\mu]_n|, |[\tilde q_t^\mu]_n|) \le L_t < \infty
	\label{eq:optilde-PQR-bounded-mu}
\end{align}
for all $n\in\N$, in which $[\tilde p_t^\mu]_n, [\tilde q_t^\mu]_n\in\R$ are as per \er{eq:eig-pqr-tilde-mu}. Consequently, $\optilde{P}_t^\mu, \optilde{Q}_t^\mu\in\bo(\cX_1)$.
\vspace{1mm}
\end{lemma}
\begin{proof}
Fix $\mu\in[0,1]$, $t\in\Omega^\mu$, and $n\in\N$ arbitrarily. Let $\beta\in(0,\ts{\frac{1}{2}})$ denote a fixed {\em badly approximable} number, see for example \cite{DS:12}. Define sequences $\{\rho_n^\mu\}_{n\in\N}, \{j_n^\mu\}_{n\in\Z_{\ge 0}}, \{b_n^\mu\}_{n\in\N}\subset\R_{>0}$ by
\begin{gather}
	\rho_n^\mu 
	\doteq \frac{\omega_n^\mu\, t}{\pi\, \beta},
	\qquad
	j_n^\mu \doteq \lfloor \beta\, \rho_n^\mu + \demi \rfloor,
	\qquad
	b_n^\mu \doteq \left\{ \ba{cl}
			\lfloor \rho_n^\mu \rfloor,
			& \beta\, \rho_n^\mu - j_n^\mu \in(0,\demi),
			\\
			\lceil \rho_n^\mu \rceil, 
			& \beta\, \rho_n^\mu - j_n^\mu \in(-\demi,0),
		\ea \right.
	\label{eq:seq}
\end{gather}
for all $n\in\N$. Note by definition that $j_n^\mu\in\N\cup\{0\}$ is determined by rounding $\beta\, \rho_n^\mu\in\R_{>0}$ to the nearest non-negative integer (with tie breaking towards $+\infty$). Furthermore, as $\omega_n^\mu\, t$ cannot be an integer multiple of $\ts{\frac{\pi}{2}}$ by definition of $t\in\Omega^\mu$, $\beta\, \rho_n^\mu$ cannot be any integer multiple of $\demi$, and so $\beta\, \rho_n^\mu - j_n^\mu\in(-\demi,\demi)\setminus\{0\}$. Combining these facts yields
\begin{align}
	\tan(\omega_n^\mu\, t)
	= \tan(\pi\, \beta\, \rho_n^\mu)
	& = \tan(\pi\, (\beta\, \rho_n^\mu - j_n^\mu) + j_n^\mu\, \pi)
	= \tan(\pi\, (\beta\, \rho_n^\mu - j_n^\mu)),
	\label{eq:badly-1}
\end{align}
in which $\pi\, (\beta\, \rho_n^\mu - j_n)\in(-\ts{\frac{\pi}{2}}, \ts{\frac{\pi}{2}})\setminus\{0\}$. Hence, from \er{eq:eig-pqr-tilde-mu},
\begin{align}
	& | [\tilde p_t^\mu]_n|^2
	= \frac{1}{(\omega_n^\mu)^2 \, \tan^2(\omega_n^\mu\, t)}
	= \frac{t^2}{(\pi\, \beta\, \rho_n^\mu)^2 \, f(\eps_{n}^\mu)},
	\label{eq:pre-bound}
\end{align}
in which $f(\eps) \doteq \tan^2(\eps)$, \rev{$\eps\in(-\ts{\frac{\pi}{2}}, \ts{\frac{\pi}{2}})$,}
and $\eps_{n}^\mu\doteq \pi\, (\beta\, \rho_n^\mu - j_n^\mu)$. Let $f^{(4)}$ denote the fourth derivative of $f$, and note that $f^{(4)}(\eps) \in\R_{\ge 0}$ for any $\eps\in(-\ts{\frac{\pi}{2}}, \ts{\frac{\pi}{2}})$. Hence, by Taylor's theorem,
\begin{align}
	& f(\eps_n^\mu) = (\eps_n^\mu)^2 + \ts{\frac{1}{4!}}\, f^{(4)}(\tilde \eps)\, (\eps_n^\mu)^4
	\ge (\eps_n^\mu)^2,
	\label{eq:badly-2}
\end{align}
in which $\tilde\eps$ is in the interval between $0$ and $\eps_n^\mu\in(-\ts{\frac{\pi}{2}}, \ts{\frac{\pi}{2}})$. 
Meanwhile, using definitions \er{eq:seq}, it may be shown that 
\begin{align}
	| \beta\, \rho_n^\mu - j_n^\mu |
	& \ge | \beta\, b_n^\mu - j_n^\mu | = b_n^\mu \, | \beta - \ts{\frac{j_n^\mu}{b_n^\mu}} |
	\ge \ts{\frac{C_\beta}{b_n^\mu}},
	\label{eq:badly-3}
\end{align}
in which the final inequality follows as $\beta$ is always badly approximated by $\ts{\frac{j_n^\mu}{b_n^\mu}}\in\Q_{>0}$, and $C_\beta\in\R_{>0}$ is some constant dependent only on $\beta$, see \cite{DS:12}. Hence, combining \er{eq:badly-1}, \er{eq:badly-2}, \er{eq:badly-3} in \er{eq:pre-bound} subsequently yields
\begin{align}
	| [\tilde p_t^\mu]_n|^2
	& = \frac{t^2}{(\pi\, \beta\, \rho_n^\mu)^2 \, f(\eps_{n}^\mu)}
	\le \frac{t^2}{(\pi\, \beta\, \rho_n^\mu)^2 \, (\eps_n^\mu)^2}
	\nn\\
	& 
	\le \frac{t^2}{(\pi^2\, \beta\, \rho_n^\mu)^2 \, (\ts{\frac{C_\beta}{b_n^\mu}})^2}
	= (\ts{\frac{t}{\pi^2\, \beta\, C_\beta}})^2\, (\ts{\frac{b_n^\mu}{\rho_n^\mu}})^2
	\nn\\
	&
	\le (\ts{\frac{t}{\pi^2\, \beta\, C_\beta}})^2\, (1 + \ts{\frac{1}{\rho_n^\mu}})^2
	\le (\ts{\frac{t}{\pi^2\, \beta\, C_\beta}})^2\, (1 +  \ts{\frac{\pi\, \beta}{\omega_1^1\, t}})^2.
	\nn
\end{align}
Similarly, recalling the definition \er{eq:eig-pqr-tilde-mu} of $\tilde q_n^\mu$,
\begin{align}
	& |[\tilde q_n^\mu]|^2 = \ts{\frac{1}{(\omega_n^\mu)^2}} + |[\tilde p_t^\mu]_n|^2
	\le \ts{\frac{1}{(\omega_1^1)^2}} +  (\ts{\frac{t}{\pi^2\, \beta\, C_\beta}})^2\, (1 +  \ts{\frac{\pi\, \beta}{\omega_1^1\, t}})^2.
	\nn
\end{align}
Hence, defining $L_t \doteq \ts{\frac{1}{(\omega_1^1)^2}} +  (\ts{\frac{t}{\pi^2\, \beta\, C_\beta}})^2\, (1 +  \ts{\frac{\pi\, \beta}{\omega_1^1\, t}})^2$  yields \er{eq:optilde-PQR-bounded-mu}, while the application of these bounds in the definitions \er{eq:optilde-PQR-mu}, \er{eq:eig-pqr-tilde-mu} of $\optilde{P}_t^\mu$, $\optilde{Q}_t^\mu$ yields $\optilde{P}_t^\mu, \optilde{Q}_t^\mu\in\bo(\cX_1)$.
\end{proof}

In order to apply $\wt{G}_t^\mu$, it is crucial to show that the operators $\optilde{P}_t^\mu$, $\optilde{Q}_t^\mu$ can be propagated to arbitrary longer horizons in $\Omega^\mu$ via concatenations of horizons. This can be achieved using standard Schur complement operations.
\begin{lemma}
\label{lem:sum-of-angles}
Given any $\mu\in\rev{(0,1]}$, \rev{$s,\sigma\in\Omega^\mu$,} 
\begin{gather}
	\optilde{P}_s^\mu, \, \optilde{P}_{s+\sigma}^\mu, \, (\optilde{P}_s^\mu)^{-1},\,
		\optilde{Q}_s^\mu, \, \optilde{Q}_{s+\sigma}^\mu, \, (\optilde{Q}_s^\mu)^{-1} \in \bo(\cX_1)\,,
	\nn\\
	[\optilde{P}_s^ \mu + \optilde{P}_\sigma^\mu]^{-1}\, \optilde{Q}_s^\mu, \, 
		[\optilde{P}_s^ \mu + \optilde{P}_\sigma^\mu]^{-1}\, \optilde{Q}_\sigma^\mu
		 \in \bo(\cX_1)\,,
	\label{eq:sum-of-angles-bounds}
	\\
	\begin{aligned}
	\optilde{P}_{s+\sigma}^\mu
	& = \optilde{P}_s^\mu - \optilde{Q}_s^\mu\, [ \optilde{P}_s^\mu + \optilde{P}_\sigma^\mu ]^{-1}\, \optilde{Q}_s^\mu\,,
	& \qquad
	\optilde{Q}_{s+\sigma}^\mu
	& = -\optilde{Q}_s^\mu\, [ \optilde{P}_s^\mu + \optilde{P}_\sigma^\mu ]^{-1}\, \optilde{Q}_\sigma^\mu\,.
	\end{aligned}
	\label{eq:sum-of-angles}
\end{gather}
\end{lemma}
\begin{proof}
Fix $\mu\in\rev{(0,1]}$, $s,\sigma\in\Omega^\mu$, and note that $s+\sigma\in\Omega^\mu$ by Lemma \ref{lem:Omega-mu-properties}.

\rev{{\em Boundedness assertions \er{eq:sum-of-angles-bounds}:}} 
\rev{Lemma \ref{lem:optilde-PQR-bounded-mu} immediately yields that 
\begin{align}
	& \optilde{P}_s^\mu, \, \optilde{P}_\sigma^\mu,\, \optilde{P}_{s+\sigma}^\mu, \, 	
		\optilde{Q}_s^\mu, \, \optilde{Q}_\sigma^\mu,\, \optilde{Q}_{s+\sigma}^\mu
		 \in \bo(\cX_1)\,.
	\nn
\end{align}
By inspection of \er{eq:omega-lambda}, \er{eq:optilde-PQR-mu}, \er{eq:eig-pqr-tilde-mu}, note that
$|([\tilde q_n^\mu]_n)^{-1}| = |\omega_n^\mu| \, |\sin(\omega_n^\mu\, t)| \le \ts{\frac{1}{\mu}}$ for all $n\in\N$. Consequently, a bounded operator of the form \er{eq:spectral-opA} is also defined by 
\begin{align}
	\optilde{R}_s^\mu\, \xi
	\doteq \sum_{n=1}^\infty ([\tilde q_n^\mu]_n)^{-1} \, \langle \xi, \, \tilde\varphi_n \rangle_1 \, \tilde\varphi_n
	\nn
\end{align}
for all $\xi\in\cX_1$, with $\|\optilde{R}_s^\mu\|_{\bo(\cX_1)} \le \frac{1}{\mu}$. Analogously to the proof of Lemma \ref{lem:op-properties}, $(\optilde{Q}_s^\mu)^{-1} \doteq \optilde{R}_s^\mu \in \bo(\cX_1)$.
\if{false}

??

. Moreover, given any $\xi\in\cX_1$,
\begin{align}
	\optilde{R}_s^\mu\, \optilde{Q}_s^\mu\, \xi
	& = \sum_{n=1}^\infty ([\tilde q_n^\mu]_n)^{-1} \left\langle \sum_{k=1}^\infty [\tilde q_s^\mu]_k \, \langle\xi,\, \tilde\varphi_k \rangle_1 \, 
		\tilde\varphi_k , \, \tilde\varphi_n \right\rangle_1 \tilde\varphi_n
	\nn\\
	&  = \sum_{n=1}^\infty \sum_{k=1}^\infty \frac{[\tilde q_s^\mu]_k}{[\tilde q_s^\mu]_n} \langle\xi,\, \tilde\varphi_k\rangle_1\, 
	\langle \tilde\varphi_k,\, \tilde\varphi_n \rangle_1\, \tilde\varphi_n
	= \sum_{n=1}^\infty \langle\xi,\, \tilde\varphi_n\rangle_1\, \tilde\varphi_n = \xi,
	\nn
\end{align}
so that 

\fi

Next, recall by definition of $s,\sigma, s+\sigma\in\Omega^\mu$ that}
\begin{align}
	\begin{gathered}
	\omega_n^\mu\, s 
	\ne  j\, (\ts{\frac{\pi}{2}}),
	\ \omega_n^\mu\, \sigma 
	\ne  j\, (\ts{\frac{\pi}{2}}),
	\ 
	\omega_n^\mu\, (s+\sigma) \ne j\, (\ts{\frac{\pi}{2}}), \ n,j\in\N,
	\end{gathered}
	\label{eq:sum-not-an-integer}
\end{align}
\rev{so that} for any $n,j\in\N$,
\begin{gather}
	|\sec(\omega_n^\mu\, s)| < \infty,
	\ \,
	|\csc(\omega_n^\mu\, s)| < \infty,
	|\cot(\omega_n^\mu\, s)| < \infty, 
	\ 
	|\tan(\omega_n^\mu\, \sigma)| < \infty,
	\label{eq:tan-finite}
	\\
	\begin{aligned}
	\tan(\omega_n^\mu\, s) + \tan(\omega_n^\mu\, \sigma)
	& = \tan(\omega_n^\mu\, (s+\sigma) - \omega_n^\mu\, \sigma) + \tan(\omega_n^\mu\, \sigma)
	\\
	& \ne \tan(j\, \pi - \omega_n^\mu\, \sigma ) + \tan(\omega_n^\mu\, \sigma) = 0.
	\end{aligned}
	\nn
\end{gather}
\rev{Define
\begin{align}
	\wh{Q}_{s,\sigma}^\mu
	& \doteq [\op{U}_s^\mu]_{11} + (\optilde{Q}_s^\mu)^{-1}\, \optilde{P}_{s+\sigma}^\mu \in \bo(\cX_1),
	\label{eq:ophat-Q}
\end{align}
in which $[\op{U}_s^\mu]_{11}\in\bo(\cX_1)$ is as per \er{eq:group-mu}, \er{eq:group-mu-elmts}, and \mbox{$\optilde{P}_{s+\sigma}^\mu,\, (\optilde{Q}_s^\mu)^{-1}\in\bo(\cX_1)$} as demonstrated above. Note further that $\wh{Q}_{s,\sigma}^\mu$ may also be represented in the form \er{eq:spectral-opA}, with
\begin{align}
	\wh{\op{Q}}_{s,\sigma}^\mu\, \xi
	& = 
	\sum_{n=1}^\infty  [\hat q_{s,\sigma}^\mu]_n\, \langle \xi, \tilde\varphi_n \rangle_1 \, \tilde\varphi_n,
	\label{eq:op-eta}
\end{align}
in which $[\hat q_{s,\sigma}^\mu]_n$ is well-defined via \er{eq:group-mu-elmts}, \er{eq:optilde-PQR-mu} by
\begin{align}
	& [ \hat q_{s,\sigma}^\mu]_n 
	\doteq \cos(\omega_n^\mu\, s) + ([\tilde q_s^\mu]_n)^{-1}\, [ \tilde p_{s+\sigma}^\mu]_n.
	\label{eq:tilde-q}
\end{align}
Observe by \er{eq:eig-pqr-tilde-mu} and standard trigonometric identities (including sum-of-angles for $\tan$) that
\begin{align}
	[\tilde p_{s+\sigma}^\mu]_n 
	& = \frac{-1}{\omega_n^\mu\, \tan(\omega_n^\mu(s+\sigma))}
	\nn\\
	& = \frac{-1}{\omega_n^\mu\, \tan(\omega_n^\mu\, s)} + \frac{1}{\omega_n^\mu\, \sin^2(\omega_n^\mu\, s)}
	\frac{\tan(\omega_n^\mu\, s) \, \tan(\omega_n^\mu\, \sigma)}{\tan(\omega_n^\mu\, s) + \tan(\omega_n^\mu\,\sigma)},
	\label{eq:tilde-p-long}
\end{align}
in which all terms are finite \rev{by} \er{eq:sum-not-an-integer}, \er{eq:tan-finite}.
Substituting \er{eq:tilde-p-long} in \er{eq:tilde-q} subsequently yields
\begin{align}
	[ \hat q_{s,\sigma}^\mu]_n 
	& =  \cos(\omega_n^\mu\, s) + \omega_n^\mu\, \sin(\omega_n^\mu\, s) \, [ \tilde p_{s+\sigma}^\mu]_n
	\nn\\
	& 
	= \frac{\sec(\omega_n^\mu\, s)\, \tan(\omega_n^\mu\, \sigma)}{\tan(\omega_n^\mu\, s) + \tan(\omega_n^\mu\, \sigma)}
	= -\frac{[\tilde q_s^\mu]_n}{[\tilde p_s^\mu]_n + [\tilde p_\sigma^\mu]_n},
	\nn
\end{align}
in which all terms are again finite \rev{by} \er{eq:sum-not-an-integer}, \er{eq:tan-finite}. 
Hence, recalling \er{eq:ophat-Q}, \er{eq:op-eta}, it follows that $\wh{Q}_{s,\sigma}^\mu \equiv - [\optilde{P}_s^\mu + \optilde{P}_\sigma^\mu]^{-1} \, \optilde{Q}_s^\mu$, so that
\begin{align}
	& - [\optilde{P}_s^\mu + \optilde{P}_\sigma^\mu]^{-1} \, \optilde{Q}_s^\mu,
	\, - [\optilde{P}_s^\mu + \optilde{P}_\sigma^\mu]^{-1} \, \optilde{Q}_\sigma^\mu
	\in\bo(\cX_1),
	\nn
\end{align}
and boundedness follows by \er{eq:ophat-Q}. Therefore, \er{eq:sum-of-angles-bounds} holds. }

\rev{{\em Semigroup properties \er{eq:sum-of-angles}:}
Observe by \er{eq:optilde-PQR-mu}, \er{eq:eig-pqr-tilde-mu}, \er{eq:ophat-Q}, \er{eq:op-eta} that
}
\begin{align}
	& (\optilde{P}_s^\mu - \optilde{Q}_s^\mu\, [ \optilde{P}_s^\mu + \optilde{P}_\sigma^\mu ]^{-1}\, \optilde{Q}_s^\mu)\, \xi
	\rev{\, = (\optilde{P}_s^\mu - \optilde{Q}_s^\mu\, \wh{\op{Q}}_{s,\sigma}^\mu)\, \xi}
	= \sum_{n=1}^\infty [\hat p_{s,\sigma}^\mu]_n \langle \xi,\tilde\varphi_n \rangle_1\, \tilde\varphi_n\,,
	\label{eq:sum-of-angles-1}
\end{align}
where $[\hat p_{s,\sigma}^\mu]_n\in\R$, $n\in\N$, is well-defined via \er{eq:sum-not-an-integer}, \er{eq:tan-finite}\rev{, \er{eq:tilde-p-long} by}
\begin{align}
	[\hat p_{s,\sigma}^\mu]_n
	& \doteq [\tilde p_s^\mu]_n - [\tilde q_s^\mu]_n^2 ( [\tilde p_s^\mu]_n + [\tilde p_\sigma^\mu]_n )^{-1}
	\nn\\
	& = \frac{-1}{\omega_n^\mu\, \tan(\omega_n^\mu\, s)} + \frac{1}{\omega_n^\mu\, \sin^2(\omega_n^\mu\, s)}
	\frac{\tan(\omega_n^\mu\, s) \, \tan(\omega_n^\mu\, \sigma)}{\tan(\omega_n^\mu\, s) + \tan(\omega_n^\mu\,\sigma)}
	\rev{\, = [\tilde p_{s+\sigma}^\mu]_n\,.}
\end{align}
Hence, \rev{recalling \er{eq:optilde-PQR-mu}, \er{eq:eig-pqr-tilde-mu}, \er{eq:sum-of-angles-1}} yields the first equality in \er{eq:sum-of-angles}. A similar calculation (involving sum-of-angles for $\sin$) yields the second equality in \er{eq:sum-of-angles}.
\end{proof}

\begin{remark}
In the proof of Lemma \ref{lem:sum-of-angles}, and in particular the boundness property \er{eq:ophat-Q}, note that $(\optilde{Q}_s^0)^{-1}\not\in\bo(\cX_1)$, so that $\wh{\op{Q}}_{s,\sigma}^0\not\in\bo(\cX_1)$.
\hfill{$\square$}
\end{remark}

\begin{lemma}
\label{lem:G-stat-single}
Given $\mu\in(0,1]$, and any $s,\sigma\in\Omega^\mu$, 
\begin{align}
	\wt{G}_{s+\sigma}^\mu(\xi,\zeta)
	& = \stat_{\eta\in\cX_1} \{ \wt{G}_s^\mu(\xi,\eta) + \wt{G}_\sigma^\mu(\eta,\zeta) \}
	\label{eq:G-stat-single}
\end{align}
for all $\xi,\zeta\in\cX_1$, in which $\wt{G}_s^\mu$ is as per \er{eq:G-mu-long}. Furthermore,
\begin{align}
	\eta^* & \doteq - [\optilde{P}_s^\mu + \optilde{P}_\sigma^\mu]^{-1} \, ( \optilde{Q}_s^\mu\, \xi + \optilde{Q}_\sigma^\mu\, \zeta )
	\label{eq:eta-star}
\end{align}
is well-defined and satisfies
\begin{align}
	& \eta^* \in \argstat_{\eta\in\cX_1} \{ \wt{G}_s^\mu(\xi,\eta) + \wt{G}_\sigma^\mu(\eta,\zeta) \}.
	\label{eq:eta-star-argstat}
\end{align}
\end{lemma}
\begin{proof}
Given $\mu\in(0,1]$, fix any $s,\sigma\in\Omega^\mu$. \rev{Fix any $\xi,\zeta\in\cX_1$. Applying Lemma \ref{lem:sum-of-angles}, 
$$
	[\optilde{P}_s^\mu + \optilde{P}_\sigma^\mu]^{-1}\, \optilde{Q}_s^\mu, \,
	[\optilde{P}_s^\mu + \optilde{P}_\sigma^\mu]^{-1}\, \optilde{Q}_\sigma^\mu\in\bo(\cX_1)\,,
$$
so that $\eta^*$ is well-defined by \er{eq:eta-star}. By inspection of \er{eq:G-mu-long}, 
\begin{align}
	\grad_\eta \{ \wt{G}_s^\mu(\xi,\eta) + \wt{G}_\sigma^\mu(\eta,\zeta) \}
	& 
	= [\optilde{P}_s^\mu + \optilde{P}_\sigma^\mu] \, \eta +  \optilde{Q}_s^\mu\, \xi + \optilde{Q}_\sigma^\mu\, \zeta
	\nn
\end{align}
for all $\xi,\eta,\zeta\in\cX_1$, so that $0 = \grad_\eta \{ \wt{G}_s^\mu(\xi,\eta) + \wt{G}_\sigma^\mu(\eta,\zeta) \} \bigl|_{\eta = \eta^*}$. Hence, $\eta^*$ also satisfies \er{eq:eta-star-argstat}, and \er{eq:G-stat-single} subsequently follows by \er{eq:stat}.
}
\end{proof}

\begin{theorem}
\label{thm:G-mu-long-concat}
Given any $\mu\in(0,1]$, $t\in\Omega^\mu\cap[\bar t^\mu,\infty)$, and $n_t\in\N$ sufficiently large such that $\tau\doteq t/n_t\in(0,\bar t^\mu)$, the long horizon extension $\wt{G}_t^\mu$ of $G_\tau^\mu$, see \er{eq:G-mu-long}, \er{eq:G-mu}, satisfies
\begin{align}
	\wt{G}_t^\mu(\xi,\zeta)
	& = \stat_{\eta\in(\cX_1)^{n_t-1}} \left\{ 
							G_\tau^\mu(\xi,\eta_1) + \sum_{k=2}^{n_t-1} G_\tau^\mu(\eta_{k-1},\eta_k)
								+ G_\tau^\mu(\eta_{n_t-1},\zeta)
						\right\}
	\label{eq:G-mu-long-n-concat}
\end{align}
for all $\xi,\zeta\in\cX_1$, in which $(\cX_1)^{n_t-1}$ denotes the product space $\cX_1\times \cdots \times \cX_1$, $n_t-1$ times. Furthermore,
\begin{align}
	\wt{G}_t^\mu(\xi,\zeta)
	& = \stat_{\eta\in\cX_1} \left\{ \wt{G}_{k\, \tau}^\mu(\xi,\eta) + \wt{G}_{(n_t - k)\, \tau}^\mu(\eta,\zeta) \right\}
	\nn\\
	& = \wt{G}_{k\, \tau}^\mu(\xi,\eta_k^*) + \wt{G}_{(n_t - k)\, \tau}^\mu(\eta_k^*,\zeta), \quad k\in\N_{< n_t},
	\label{eq:G-mu-long-1-concat}
\end{align}
in which the $\stat$ is achieved at $\eta_k^*\in\cX_1$, where
\begin{align}
	& \eta_k^* 
	\doteq -[\optilde{P}_{k\, \tau}^\mu + \optilde{P}_{(n_t - k)\, \tau}^\mu]^{-1}\, 
	(\optilde{Q}_{k\, \tau}^\mu\, \xi + \optilde{Q}_{(n_t - k)\, \tau}^\mu\, \zeta),
	\label{eq:eta-k-star}
\end{align}
for all $k\in\N_{< n_t}$.
\end{theorem}

\begin{proof}
Fix $\mu\in(0,1]$, $t\in\Omega^\mu\cap[\bar t^\mu,\infty)$, and $n_t\in\N$, $\tau\in(0,\bar t^\mu)$ as per the theorem statement. By Lemma \ref{lem:Omega-mu-properties}, $k \, \tau\in\Omega^\mu$ for all $k\in\N$. 
Hence, given $k\in[2,n_t]\cap\N$, applying Lemma \ref{lem:G-stat-single} with $s\doteq (k-1)\, \tau$ 
and $\sigma \doteq \tau$ yields
\begin{align}
	\wt{G}_{k\, \tau}^\mu(\xi,\zeta)
	& = \stat_{\eta_{k-1}\in\cX_1} \{ \wt{G}_{(k-1)\, \tau}^\mu(\xi,\eta_{k-1}) + \wt{G}_{\tau}^\mu(\eta_{k-1},\zeta) \}
	\nn\\
	& = \stat_{\eta_{k-1}\in\cX_1} \{ \wt{G}_{(k-1)\, \tau}^\mu(\xi,\eta_{k-1}) +G_{\tau}^\mu(\eta_{k-1},\zeta) \}
	\nn\\
	& =  \stat_{\eta_{k-1}, \eta_{k-2}\in\cX_1} \{ \wt{G}_{(k-2)\, \tau}^\mu(\xi,\eta_{k-2}) + G_{\tau}^\mu(\eta_{k-2},\eta_{k-1})  
						+G_{\tau}^\mu(\eta_{k-1},\zeta) \},
	\nn
\end{align}
which yields \er{eq:G-mu-long-n-concat} by induction, for $k=n_t$. Again applying Lemma \ref{lem:G-stat-single} with $s\doteq j\, \tau$, $\sigma\doteq (n_t - j)\, \tau$ for $j\in\N_{<n_t}$ subsequently yields \er{eq:G-mu-long-1-concat}, \er{eq:eta-k-star}.
\end{proof}

In view of \er{eq:idem-rep}, \er{eq:G-mu-long}, \er{eq:G-mu-long-and-short}, and Theorem \ref{thm:G-mu-long-concat}, $W_t^\mu$ of \er{eq:idem-rep} may be generalized to $\wt{W}_t^\mu:\cX_1\rightarrow\ol{\R}$ for $t\in\Omega^\mu$ via
\begin{align}
	\wt{W}_t^\mu(\xi)
	& \doteq \stat_{\zeta\in\cX_1} \{ \wt{G}_t^\mu(\xi,\zeta) + \rev{\psi}(\zeta) \}
	\label{eq:idem-rep-long}
\end{align}
for all $\xi = \xi_0\in\cX_1$. \rev{With $\psi = \psi_v$ as per \er{eq:psi-v},} selecting $n_t\in\N$ as indicated, and generalising \er{eq:z-star}, note that the $\stat$ in \er{eq:idem-rep-long} is achieved at
\begin{align}
	\zeta_{\xi_0}^* 
	& = -(\optilde{P}_t^\mu)^{-1} ( \optilde{Q}_t^\mu\, \xi_0 + \op{E}_\mu^{-1}\, v ).
	\label{eq:zeta-0-star}
\end{align}
Applying \er{eq:G-mu-long-1-concat}, \er{eq:eta-k-star}, note further that
\begin{align}
	\eta_1^*
	& = - [\optilde{P}_\tau^\mu + \optilde{P}_{(n_t-1)\tau}^\mu ]^{-1}
	( \optilde{Q}_\tau^\mu\, \xi_0 + \optilde{Q}_{(n_t-1)\tau}^\mu\, \zeta_{\xi_0}^* ).
	\nn
\end{align}
Motivated by \er{eq:pi-0}, define (for the long horizon case)
\begin{align}
	& \pi_0
	\doteq \op{E}_\mu \, \grad \wt{W}_t(\xi_0)
	= \op{E}_\mu \, ( \optilde{P}_\tau^\mu\, \xi_0 + \optilde{Q}_{\tau}^\mu\, \eta_1^* )
	\nn\\
	& \!\! = \op{E}_\mu ( \optilde{P}_\tau^\mu - \optilde{Q}_{\tau}^\mu \, 
	[\optilde{P}_\tau^\mu + \optilde{P}_{(n_t-1)\tau}^\mu ]^{-1} \optilde{Q}_\tau^\mu)\, \xi_0
	- \op{E}_\mu\, \optilde{Q}_{\tau}^\mu \, 
	[\optilde{P}_\tau^\mu + \optilde{P}_{(n_t-1)\tau}^\mu ]^{-1} \optilde{Q}_{(n_t-1)\tau}^\mu \, \zeta_{\xi_0}^*
	\nn\\
	& \!\! = \op{E}_\mu \, ( \optilde{P}_t^\mu\, \xi_0 + \optilde{Q}_t^\mu\, \zeta_{\xi_0}^*)
	\label{eq:pi-0-TPBVP}
	= \op{E}_\mu \, ( \optilde{P}_t^\mu - \optilde{Q}_t^\mu\, (\optilde{P}_t^\mu)^{-1}\, \optilde{Q}_t^\mu) \, \xi_0 
	- \op{E}_\mu\, \optilde{Q}_t^\mu\, (\optilde{P}_t^\mu)^{-1}\, \optilde{E}_\mu^{-1}\, \pi_t,
\end{align}
in which the second last and last equalities follow by Lemma \ref{lem:sum-of-angles} and \er{eq:zeta-0-star}. This is of exactly the same form as \er{eq:pi-0}. 

In this way, \er{eq:prototype-mu}, \er{eq:prototype-mu-elmts} extend to all horizons in $\Omega^\mu$, which is dense in $\R_{>0}$.
By way of the action principle, it may be noted that the concatenated trajectory defined by \er{eq:idem-rep-long} renders the payoff \er{eq:payoff-mu} stationary, implying that it is a solution of the approximate (right-hand) wave equation in \er{eq:wave-mu}. 

Explicitly,
$\wh{\op{U}}_t^\mu$ of \er{eq:prototype-mu}, \er{eq:prototype-mu-elmts} extends to yield corresponding elements of a long horizon prototype $\{\optilde{U}_t^\mu\}_{t\in\Omega^\mu}$, with
\begin{align}
	& \hspace{-3mm}
	\left( \ba{c}
		\xi_t
		\\
		\pi_t
	\ea \right)
	= \optilde{U}_t^\mu
		\left( \ba{c}
		\xi_0
		\\
		\pi_0
	\ea \right), \quad
	\optilde{U}_t^\mu
	\doteq
	 \left( \ba{c|c}
	 	[\optilde{U}_t^\mu]_{11} & [\optilde{U}_t^\mu]_{12}
		\\[0mm]
		& \\[-3mm]\hline
		& \\[-3mm]
	 	[\optilde{U}_t^\mu]_{21} & [\optilde{U}_t^\mu]_{22}
	\ea \right)\!, \quad t\in\Omega^\mu,
	\label{eq:prototype-mu-long}
\end{align}
in which $[\optilde{U}_t^\mu]_{11}\in\bo(\cX_1)$, $[\optilde{U}_t^\mu]_{12}\in\bo(\cX;\cX_1)$, $[\optilde{U}_t^\mu]_{21}\in\bo(\cX_1;\cX)$, $[\optilde{U}_t^\mu]_{22}\in\bo(\cX)$ are given by
\begin{equation}
	\begin{aligned}
	{[\optilde{U}_t^\mu]}_{11}
	& \doteq 
	-(\optilde{Q}_t^\mu)^{-1}\, \optilde{P}_t^\mu,
	&
	[\optilde{U}_t^\mu]_{12}
	& \doteq (\optilde{Q}_t^\mu)^{-1}\, \op{E}_\mu^{-1},
	\\
	[\optilde{U}_t^\mu]_{21}
	& \doteq \rev{-} \op{E}_\mu\, \optilde{Q}_t^\mu\, \left( \op{I} - [(\optilde{Q}_t^\mu)^{-1} \, \optilde{P}_t^\mu]^2 \right),
	&
	[\optilde{U}_t^\mu]_{22}
	& \doteq - \op{E}_\mu\, \optilde{P}_t^\mu\, (\optilde{Q}_t^\mu)^{-1}\, \op{E}_\mu^{-1},
	\end{aligned}
	\label{eq:prototype-mu-elmts-long}
\end{equation}
for all $t\in\Omega^\mu$. Recalling \er{eq:optilde-PQR-mu}, \er{eq:eig-pqr-tilde-mu}, the corresponding long horizon extension of Lemma \ref{lem:op-properties} implies that these operators exhibit the spectral representation \er{eq:spectral-opA}, with corresponding eigenvalues given by
\begin{align}
	\begin{aligned}
	\lbrack [\tilde u_t^\mu]_{11} \rbrack_n
	& \doteq -\frac{[\tilde p_t^\mu]_n}{[\tilde q_t^\mu]_n} = \cos(\omega_n^\mu\, t),
	&
	[[\tilde u_t^\mu]_{12}]_n 
	& \doteq \sin(\omega_n^\mu\, t),
	\\
	[[\tilde u_t^\mu]_{21}]_n
	& \doteq -\sin(\omega_n^\mu\, t),
	& 
	[[\tilde u_t^\mu]_{22}]_n
	& \doteq \cos(\omega_n^\mu\, t).
	\end{aligned}
	\label{eq:prototype-mu-long-eig}
\end{align}
The prototype \er{eq:prototype-mu-elmts-long} is extended to negative horizons $t$, that is $-t\in\Omega^\mu$, via 
\begin{align}
	\optilde{U}_t^\mu
	& \doteq \optilde{U}_{-t}^\mu = 
	\left( \ba{c|c}
	 	[\optilde{U}_{-t}^\mu]_{11} & [\optilde{U}_{-t}^\mu]_{12}
		\\[0mm]
		& \\[-3mm]\hline
		& \\[-3mm]
	 	[\optilde{U}_{-t}^\mu]_{21} & [\optilde{U}_{-t}^\mu]_{22}
	\ea \right)\!, \quad -t\in\Omega^\mu.
	\nn
\end{align}

\begin{theorem}
\label{thm:group-rep-mu}
The set $\{\wt{\op{U}}_t^\mu,\, \wt{\op{U}}_{-t}^\mu \}_{t\in\Omega^\mu}$ of \er{eq:prototype-mu-long}, \er{eq:prototype-mu-elmts-long} defines a uniformly continuous group, and is equivalent to the subgroup $\{\op{U}_t^\mu, \, \op{U}_{-t}^\mu\}_{t\in\Omega^\mu}$ of \er{eq:group-mu}, \er{eq:group-mu-elmts} generated by $\op{A}^\mu$ of \er{eq:generators}.
\end{theorem}
\begin{proof}
Immediate by comparison of \er{eq:group-mu}, \er{eq:group-mu-elmts} with \er{eq:prototype-mu-long}, \er{eq:prototype-mu-elmts-long}, \er{eq:prototype-mu-long-eig}.
\end{proof}

It remains to establish convergence of the approximating long horizon group $\{\optilde{U}_t\}_{t\in\Omega^\mu}$ to the subgroup $\{\op{U}_t\}_{t\in\Omega^0}$ as $\mu\rightarrow 0$. To this end, fix any sequence
\begin{align}
	& S\doteq \{\mu_i\}_{i\in\N}\subset\R_{>0},
	\quad \text{\it s.t.} \quad \lim_{i\rightarrow\infty} \mu_i = 0,
	\label{eq:sequence-S}
\end{align}
and define the common set of horizons $\ol{\Omega}^S\subset\R_{>0}$ by
\begin{align}
	\ol{\Omega}^S
	& \doteq \Omega^0\cap\bigcap_{i\in\N} \Omega^{\mu_i}
	= \R_{>0} \setminus \{ t_{n,j,k}^0,\, t_{n,j,k}^{\mu_i} \}_{j,k,i\in\N, n\in\N\cup\{\infty\}}.
	\label{eq:Omega-bar}
\end{align}
As $\{ t_{n,j,k}^0 \}_{j,k\in\N, n\in\N\cup\{\infty\}}$, $\{t_{n,j,k}^{\mu_i} \}_{j,k,i\in\N, n\in\N\cup\{\infty\}}$ define countable subsets of $\R_{>0}$, their union is also countable and has zero measure. Hence, $\ol{\Omega}^S$ is an uncountable subset of $\R_{>0}$, with the same measure as $\R_{>0}$.

\begin{theorem}
\label{thm:convergence}
Given any sequence $S$ as per \er{eq:sequence-S}, and any horizon $t\in\ol{\Omega}^S$ as per \er{eq:Omega-bar},
\begin{align}
	0 & = 
	\lim_{i\rightarrow\infty} \left\| \, \optilde{U}_t^{\mu_i} \! \left( \ba{c} x \\ p \ea \right) - \op{U}_t  \left( \ba{c} x \\ p \ea \right)  \right\|_\sharp
	\label{eq:convergence}
\end{align}
for all $(x,p)\in\cY \equiv \cX_1\times\cX$.
\end{theorem}
\begin{proof}
With $t\in\ol{\Omega}^S$ fixed as per the hypothesis, note that $t\in\Omega^0$, and $t\in\Omega^{\mu_i}$ for all $i\in\N$. Hence, the result follows by assertion {\em (vi)} of Theorem \ref{thm:generators}.
\end{proof}


\section{Application to solving a TPBVP}
\label{sec:TPBVP}

Given $x,z\in\cX_2$, $\ol{\cX_2} = \cX_1$, and $t\in\Omega^0$, consider the TPBVP defined with respect to $(t,x,z)$ by
\begin{align}
	& \textsf{(TPBVP)} \quad
	\left\{ \ba{c}
	\text{Find $p = \dot{x}(0)\in\cX$}
	\\
	\text{s.t. \er{eq:wave} holds with}
	\\
	\text{$x(0) = x$, $x(t) = z$.}
	\ea \right.
	\label{eq:TPBVP}
\end{align}
\rev{
The group $\{\op{U}_t\}_{t\in\ol{\Omega}^S}$ generated by $\op{A}$, see \er{eq:generators}, explicitly propagates solutions of \er{eq:wave} for any initial data. Its construction also facilitates the solution of TPBVPs constrained by \er{eq:wave}, including \er{eq:TPBVP}. 
}

\subsection{Solution and convergence}
The solution of TPBVP \er{eq:TPBVP} can be approximated via the long horizon subgroup $\{ \optilde{U}_t^\mu \}_{t\in\Omega^\mu}$. In particular, the solution $p\in\cX$ may be approximated by $\pi_0$ in \er{eq:pi-0-TPBVP} by setting $\xi_0 = x$ and replacing the achieved terminal state $\zeta_{\xi_0}^*$ with the desired terminal state $z$. That is, the approximate and expected exact solutions are given by
\begin{align}
	& \pi_0^{\mu_i}
	= \op{E}_{\mu_i} ( \optilde{P}_t^{\mu_i}\, x + \optilde{Q}_t^{\mu_i}\, z ),
	\quad
	\pi_0^0
	= \opA ( \optilde{P}_t^{0}\, x + \optilde{Q}_t^{0}\, z ),
	\label{eq:pi-0-both}
\end{align}
in which $\mu_i$ is an element of sequence $S$ of \er{eq:sequence-S}, and $t\in\Omega^0$ is replaced with an arbitrarily close $t\in\ol{\Omega}^S$. The corresponding solutions of the wave equation \er{eq:wave} and its approximation \er{eq:wave-mu} are defined via the respective groups of Theorem \ref{thm:generators} applied to the same initial conditions $(x,\pi_0^{\mu_i})$. In particular, these solutions are given by
\begin{align}
	& \hspace{-3mm}
	\left( \ba{c}
		x_s^{\mu_i}
		\\
		p_s^{\mu_i}
	\ea \right)
	= \op{U}_s \left( \ba{c}
		x
		\\
		\pi_0^{\mu_i}
	\ea \right)\!,
	\qquad
	\left( \ba{c}
		\xi_s^{\mu_i}
		\\
		\pi_s^{\mu_i}
	\ea \right)
	= \optilde{U}_s^{\mu_i} \left( \ba{c}
		x
		\\
		\pi_0^{\mu_i}
	\ea \right)\!,	
	\label{eq:TPBVP-solutions}
\end{align}
for all $s\in(0,t)\cap\ol{\Omega}^S$, $i\in\N$.
In order to show that these solutions converge, and satisfy TPBVP \er{eq:TPBVP}, some preliminary convergence results are required.

%
%
\begin{lemma}
\label{lem:op-PQR-convergence}
Given any $t\in\ol{\Omega}^S$, $\xi\in\cX_1$,
\begin{align}
	\lim_{i\rightarrow\infty} \| \optilde{P}_t^{\mu_i}\, \xi - \optilde{P}_t^0\, \xi \|_1
	& = 0 = \lim_{i\rightarrow\infty} \| \optilde{Q}_t^{\mu_i}\, \xi - \optilde{Q}_t^0\, \xi \|_1.
	\label{eq:op-PQR-convergence}
\end{align}
\end{lemma}
\begin{proof}
Fix any $t\in\ol{\Omega}^S$, $\xi\in\cX_1$. Lemma \ref{lem:optilde-PQR-bounded-mu} implies that $\optilde{P}_t^\mu, \optilde{Q}_t^\mu\in\bo(\cX_1)$ for all $\mu\in S\cup\{0\}$. Hence, $\optilde{P}_t^{\mu_i}\, \xi - \optilde{P}_t^0\, \xi \in \cX_1$, and $\langle \optilde{P}_t^{\mu_i}\, \xi - \optilde{P}_t^0\, \xi\,, \cdot \rangle_1:\cX_1\rightarrow\R$ is closed for every $i\in\N$. Consequently, recalling \er{eq:optilde-PQR-mu},
\begin{align}
	& \| \optilde{P}_t^{\mu_i}\, \xi - \optilde{P}_t^0\, \xi \|_1^2
	= \left\langle \optilde{P}_t^{\mu_i}\, \xi - \optilde{P}_t^0\, \xi, \,
	\sum_{n=1}^\infty ([\tilde p_t^{\mu_i}]_n - [\tilde p_t^0]_n) \, \langle \xi, \tilde\varphi_n \rangle_1 \tilde\varphi_n
	\right\rangle_1
	\nn\\
	&
	= \sum_{n=1}^\infty ([\tilde p_t^{\mu_i}]_n - [\tilde p_t^0]_n) \, \langle \xi, \tilde\varphi_n \rangle_1\, 
	\left\langle \optilde{P}_t^{\mu_i}\, \xi - \optilde{P}_t^0\, \xi, \tilde\varphi_n \right\rangle_1
	\nn\\
	&
	= \sum_{n=1}^\infty |[\tilde p_t^{\mu_i}]_n - [\tilde p_t^0]_n|^2 \, | \langle \xi, \tilde\varphi_n \rangle_1 |^2.
	\label{eq:star}
\end{align}
Define the sequences $\{\alpha_n^i\}_{n,i\in\N}$, $\{\beta_n\}_{n\in\N}$ by
\begin{align}
	\alpha_n^i 
	& \doteq |[\tilde p_t^{\mu_i}]_n - [\tilde p_t^0]_n|^2 \, | \langle \xi, \tilde\varphi_n \rangle_1 |^2,
	\ \
	\beta_n
	\doteq 4\, L_t^2\, | \langle \xi, \tilde\varphi_n \rangle_1 |^2,
	\nn
\end{align}
for all $n,i\in\N$. The triangle inequality and Lemma \ref{lem:optilde-PQR-bounded-mu} yield
\begin{align}
	& 0 \le \alpha_n^i 
	\le 2\, (|[\tilde p_t^{\mu_i}]_n^2 + |[\tilde p_t^0]_n^2 ) | \langle \xi, \tilde\varphi_n \rangle_1 |^2
	\le \beta_n,
	\quad n,i\in\N,
	\nn\\
	& 0\le \sum_{n=1}^\infty \beta_n 
	= 4\, L_t^2 \sum_{n=1}^\infty | \langle \xi, \tilde\varphi_n \rangle_1 |^2
	= 4\, L_t^2\, \|\xi\|_1^2 < \infty.
	\label{eq:pre-dominated-cgce}
\end{align}
Meanwhile, recalling \er{eq:eig-pqr-tilde-mu}, note in the definition of $\alpha_n^i$ that
\begin{align}
	& | [\tilde p_t^{\mu_i}]_n - [\tilde p_t^0]_n |
	= \left| \frac{-1}{\omega_n^{\mu_i}\, \tan(\omega_n^{\mu_i}\, t)} + \frac{1}{\omega_n^0\, \tan(\omega_n^0\, t)} \right|
	\nn\\
	& \le \left| \frac{-1}{\omega_n^{\mu_i}\, \tan(\omega_n^{\mu_i}\, t)} + \frac{1}{\omega_n^{0}\, \tan(\omega_n^{\mu_i}\, t)}  \right|
	+
	\left| \frac{1}{\omega_n^0\, \tan(\omega_n^0\, t)} - \frac{1}{\omega_n^{0}\, \tan(\omega_n^{\mu_i}\, t)} \right|
	\nn\\
	& = \ts{\frac{1}{\omega_n^0}}\, | \tilde p_n^{\mu_i} |\, |\omega_n^{\mu_i} - \omega_n^0 | + \ts{\frac{1}{\omega_n^0}}\, 
	| \cot(\omega_n^0\, t) - \cot(\omega_n^{\mu_i}\, t) |.
	\nn
\end{align}
As $t\in\ol{\Omega}^S$, the map $\omega\mapsto\cot(\omega\, t)$ is continuous for all $\omega$ in a sufficiently small neighbourhood of $\omega_n^0$, $n\in\N\cup\{\infty\}$ fixed. Hence, as $\{\omega_n^{\mu_i}\}_{i\in\N}$ defines a convergent sequence with limit $\omega_n^0\in\R_{>0}$, and $|[\tilde p_t^{\mu_i}]_n|\le L_t$ for all $i\in\N$ by Lemma \ref{lem:optilde-PQR-bounded-mu}, the above inequality implies that
$\lim_{i\rightarrow\infty}  | [\tilde p_t^{\mu_i}]_n - [\tilde p_t^0]_n | \le 0$, so that from the preceding definition of $\alpha_n^i$,
\begin{align}
	& \lim_{i\rightarrow\infty} \alpha_n^i = \alpha_n \doteq 0,
	\qquad n\in\N.
	\label{eq:lim-alpha}
\end{align}
Hence, the Dominated Convergence Theorem \cite[p.77]{R:88}, \er{eq:star}, \er{eq:pre-dominated-cgce}, and \er{eq:lim-alpha} imply that
\begin{align}
	0 & = \sum_{n=1}^\infty \alpha_n 
	= \lim_{i\rightarrow\infty} \sum_{n=1}^\infty \alpha_n^i
	= \lim_{i\rightarrow\infty} \sum_{n=1}^\infty |[\tilde p_t^{\mu_i}]_n - [\tilde p_t^0]_n|^2 \, | \langle \xi, \tilde\varphi_n \rangle_1 |^2
	\nn\\
	& = \lim_{i\rightarrow\infty} \| \optilde{P}_t^{\mu_i}\, \xi - \optilde{P}_t^0\, \xi \|_1,
	\nn
\end{align}
as required. A similar argument yields the corresponding right-hand limit in \er{eq:op-PQR-convergence}.
\end{proof}

\begin{lemma}
\label{lem:pi-0-convergence}
Given any $t\in\ol{\Omega}^S$, $x,z\in\cX_2$, 
\begin{align}
	& 0 = \lim_{i\rightarrow\infty} \| \pi_0^{\mu_i} - \pi_0^0 \|,
	\label{eq:pi-0-convergence}
\end{align}
where $\pi_p^{\mu_i}$, $\pi_0^0$ are as per \er{eq:pi-0-both}.
\end{lemma}
\begin{proof}
Fix $t\in\ol{\Omega}^S$, $x,z\in\cX_2$. 
By adding and subtracting terms via \er{eq:pi-0-both}, and recalling the definition of $\op{E}_\mu$ in \er{eq:op-E}, note that
\begin{align}
	& \pi_0^{\mu_i} - \pi_0^0
	= 
	\op{E}_{\mu_i} \, [ \optilde{P}_t^{\mu_i} - \optilde{P}_t^0] \, x 
	+ \op{E}_{\mu_i} \, [ \optilde{Q}_t^{\mu_i} - \optilde{Q}_t^0]\, z
	+ \opAsqrt\, (\op{I}_{\mu_i}^\half - \op{I})\, \opAsqrt \, [  \optilde{P}_t^0 \, x 
	+ \optilde{Q}_t^0 \, z ]\, .
	\nn
\end{align}
Commuting $\opAsqrt$ with $\op{I}_{\mu_i}^\half$, $\optilde{P}_t^{\mu_i}$, $\optilde{Q}_t^{\mu_i}$, $\optilde{P}_t^{0}$, $\optilde{Q}_t^{0}$, and applying the triangle inequality,
\begin{align}
	& \| \pi_0^{\mu_i} - \pi_0^0 \| - \| (\op{I}_{\mu_i}^\half - \op{I})\, \opAsqrt \, [ \optilde{P}_t^0 \, \opAsqrt\, x + \optilde{Q}_t^0 \, \opAsqrt\, z ] \|
	\nn\\
	&
	\le 
	\| \op{I}_{\mu_i}^\half \, \opAsqrt\, [ \optilde{P}_t^{\mu_i} - \optilde{P}_t^0] \, \opAsqrt \, x \|
	+ \| \op{I}_{\mu_i}^\half\, \opAsqrt \, [ \optilde{Q}_t^{\mu_i} - \optilde{Q}_t^0]\, \opAsqrt\, z \|
	\nn\\
	& \le \sup_{\mu\in(0,1]} \| \op{I}_{\mu}^\half \|_{\bo(\cX)} \left( \| [ \optilde{P}_t^{\mu_i} - \optilde{P}_t^0] \, \opAsqrt \, x \|_1
		+ \| [ \optilde{Q}_t^{\mu_i} - \optilde{Q}_t^0]\, \opAsqrt\, z \|_1 \right)
	\nn\\
	&
	\le \| [ \optilde{P}_t^{\mu_i} - \optilde{P}_t^0] \, \opAsqrt \, x \|_1
		+ \| [ \optilde{Q}_t^{\mu_i} - \optilde{Q}_t^0]\, \opAsqrt\, z \|_1\,,
	\nn
\end{align}
in which $\Lambda^\half x,\, \Lambda^\half z\in\cX_1$, and $\|\op{I}_\mu^\half\|_{\bo(\cX)} \le 1$ for all $\mu\in(0,1]$.
Also,
$0 = \lim_{\mu\rightarrow 0} \|(\op{I}_\mu^\half - \op{I})\, \eta\|$ for all $\eta\in\cX$, including for $\eta \doteq \opAsqrt \, [ \optilde{P}^0(t) \, \opAsqrt\, x + \optilde{Q}^0(t) \, \opAsqrt\, z ]\in\cX$. Hence, applying Lemma \ref{lem:op-PQR-convergence} immediately yields \er{eq:pi-0-convergence}.
\end{proof}

\begin{theorem}
\label{thm:TPBVP-converge}
Given any $t\in\ol{\Omega}^S$, $x,z\in\cX_2$,
\begin{align}
	& 0 = \lim_{i\rightarrow\infty} \left\| \left( \ba{c}
		\xi_s^{\mu_i}
		\\ 
		\pi_s^{\mu_i}
	\ea \right)
	- 
	\left( \ba{c}
		x_s^{\mu_i}
		\\ 
		p_s^{\mu_i}
	\ea \right)
	\right\|_\sharp
	\label{eq:TPBVP-converge}
\end{align}
for all $s\in(0,t)\cap\ol{\Omega}^S$, where $(\xi_s^{\mu_i}, \pi_s^{\mu_i})$ and $(x_s^{\mu_i}, p_s^{\mu_i})$ are as per \er{eq:TPBVP-solutions}. Furthermore, the solution to TPBVP \er{eq:TPBVP}, and a sequence of approximations to it, are given respectively by $\pi_0^0$ and $\{\pi_0^{\mu_i}\}_{i\in\N}$ as per \er{eq:pi-0-both}. 
\end{theorem}
\begin{proof}
Fix $t\in\ol{\Omega}^S$, $x,z\in\cX_2$, $s\in(0,t)\cap\ol{\Omega}^S$. Subtracting the left- from the right-hand trajectory in \er{eq:TPBVP-solutions}, 
\begin{align}
	& \!\! \left\| \left( \ba{c}
		\xi_s^{\mu_i} - x_s^{\mu_i}
		\\
		\pi_s^{\mu_i} - p_s^{\mu_i}
	\ea \right) \right\|_\sharp
	= \left\| [ \, \optilde{U}_s^{\mu_i} - \op{U}_s ] \left( \ba{c}
								x \\ \pi_0^{\mu_i}
						\ea \right) 
	\right\|_\sharp
	\nn\\
	& \!\! \le \| \optilde{U}_s^{\mu_i} -  \op{U}_s \|_{\bo(\cY)} \left\| \left( \!\! \ba{c}
									0 \\ \pi_0^{\mu_i} - \pi_0^0
								\ea \!\! \right) \right\|_\sharp \!
				+ {\left\| [ \optilde{U}_s^{\mu_i} - \op{U}_{s} ] \left( \!\! \ba{c}
											x \\ \pi_0^0
										\ea \!\! \right) \right\|_\sharp}
	\nn\\
	& \!\! \le 2\, M \exp(\omega\, s)\, \| \pi_0^{\mu_i} - \pi_0^0 \| 
				+ {\left\| [ \optilde{U}_s^{\mu_i} - \op{U}_s ] \left( \ba{c}
											x \\ \pi_0^0
										\ea \right) \right\|_\sharp}\,,
	\nn
\end{align}
in which assertions {\em (ii)} and {\em (iv)} of Theorem \ref{thm:generators} imply that 
$$
	\max(\sup_{i\in\N} \|\optilde{U}_s^{\mu_i} \|_{\bo(\cY)}, \|\op{U}_s \|_{\bo(\cY)}) \le M\, \exp(\omega\, s)
$$
for some $M,\omega\in\R_{\ge 0}$.
Taking the limit as $i\rightarrow\infty$ via Lemma \ref{lem:pi-0-convergence} and Theorem \ref{thm:convergence}, convergence of the solutions \er{eq:TPBVP-solutions} to the wave equation \er{eq:wave} and its approximation \er{eq:wave-mu} follows. 

It remains to verify that the $\{\pi_0^{\mu_i}\}_{i\in\N}$ defines a convergent sequence of solutions to TPBVP \er{eq:TPBVP}. (Note that the conclusion that $\pi_0^0$ of \er{eq:pi-0-both} describes a solution to TPBVP \er{eq:TPBVP} subsequently follows by Lemma \ref{lem:pi-0-convergence}.) To this end, recalling \er{eq:prototype-mu-long}, \er{eq:prototype-mu-elmts-long}, and Theorem \ref{thm:group-rep-mu}, note that the propagated final state $\xi_t^{\mu_i}$ of the approximating Cauchy problem given by the right-hand equations in \er{eq:Cauchy}, \er{eq:wave-mu} with $\mu = \mu_i$,
%
%
corresponding to initialisation $(x, \pi_0^{\mu_i})$, is
\begin{align}
	& \xi_t^{\mu_i}
	= [\optilde{U}_t^{\mu_i}]_{11} \, x + [\optilde{U}_t^{\mu_i}]_{11} \, \pi_0^{\mu_i}
	= -(\optilde{Q}_t^{\mu_i})^{-1} \, \optilde{P}_t^{\mu_i}\, x + (\optilde{Q}_t^{\mu_i})^{-1}\, \op{E}_{\mu_i}^{-1}\, \pi_0^{\mu_i}
	\nn\\
	& = -(\optilde{Q}_t^{\mu_i})^{-1} \, \optilde{P}_t^{\mu_i}\, x + (\optilde{Q}_t^{\mu_i})^{-1}\, \op{E}_{\mu_i}^{-1}\, \op{E}_{\mu_i}
	( \optilde{P}_t^{\mu_i}\, x + \optilde{Q}_t^{\mu_i}\, z )
	= z,
	\nn
\end{align}
as required.
\end{proof}

\rev{
%
Given $x,z,\in\cX_1$, $\mu\in S$, and $t\in\ol{\Omega}^s$, Lemma \ref{lem:pi-0-convergence} and Theorem \ref{thm:TPBVP-converge} indicate that  TPBVP \er{eq:TPBVP} has the solution $\pi_0 = \pi_0^0$, which can be approximated by $\pi_0^\mu$ as per \er{eq:pi-0-both}. 
The spectral representations \er{eq:op-properties} and \er{eq:optilde-PQR-mu} imply that 
\begin{align}
	\pi^\mu
	& = \sum_{n=1}^\infty \frac{\lambda_n}{1 + \mu^2\, \lambda_n} 
			\left[ \tilde p_n^\mu\, \langle x,\, \tilde\varphi_n \rangle_1 + \tilde q_n^\mu\, \langle z,\, \tilde\varphi_n \rangle_1 \right] \tilde\varphi_n
	\label{eq:TPBVP-sol}
\end{align}
where $\lambda_n^{-1}$ and $\tilde\varphi_n$ denote the eigenvalues and eigenvectors of $\Lambda^{-1}\in\bo(\cX)$.
}


\subsection{Example}
In order to illustrate \rev{an} application of \er{eq:TPBVP-sol}, 
\rev{select} $X \doteq [0,1]^2\subset\R^2$, and define
\begin{align}
	\begin{aligned}
	\cX & \doteq \Ltwo(X;\R)\,,
	\quad
	&& \cX_0 \doteq \funspace{H}_0^2(X;\R),
	%
	\\
	\Lambda
	& \doteq -\partial_1^2 - \partial_2^2\,,
	&& 
	\dom (\Lambda) = \cX_0\,,
	\end{aligned}
	\label{eq:ex-TPBVP}
\end{align}
\rev{where} $\partial_{1}$ and $\partial_{2}$ denote the partial derivative operators defined with respect to the first and second cartesian coordinates in $\R^2$ respectively.
It may be noted that $-\Lambda$ is the Laplacian operator on $X$, and $\Lambda$ is linear, unbounded, positive, self-adjoint, and in the possession of a compact inverse. 
%
Its eigenvalues $\lambda_{n,m}^{-1}\in\R_{>0}$ and eigenvectors $\tilde\varphi_{n,m}\in\cX_1$ are defined respectively by $\lambda_{n,m} \doteq (n^2 + m^2)\, \pi^2$ and $\tilde\varphi_{n,m}(x_1,x_2) \doteq (2/\sqrt{\lambda_{n,m}})\, \sin(n\, \pi\, x_1)\, \sin(m\, \pi\, x_2)$ for all $n,m\in\N$, $(x_1,x_2)\in X$. 

For illustrative purposes, the specific initial state $x\in\cX_0$ is chosen (arbitrarily) to be the zero function on $X$, while the terminal state $z\in\cX_0$ is selected to be as per Figure \ref{fig:ex-TPBVP-a}. A horizon $t\doteq \pi/3$ is assumed. The initial velocity $\dot x(0) = \pi_0^0$ obtained in the $\mu=0$ limit in \er{eq:TPBVP-sol} is illustrated in Figure \ref{fig:ex-TPBVP-b}. By propagating the initial state $x(0)=x$ and velocity $\dot x(0)$ forward in time, it may be seen that \er{eq:TPBVP-sol} does indeed solve the two-point boundary value problem of interest, see Figure \ref{fig:evolution}.

\begin{figure}[h]
\vspace{0mm}
\begin{center}
\psfrag{yyy}{$x_2$}
\psfrag{zzz}{$z$}
\psfrag{xxx}{$x_1$}
\includegraphics[width=70mm]{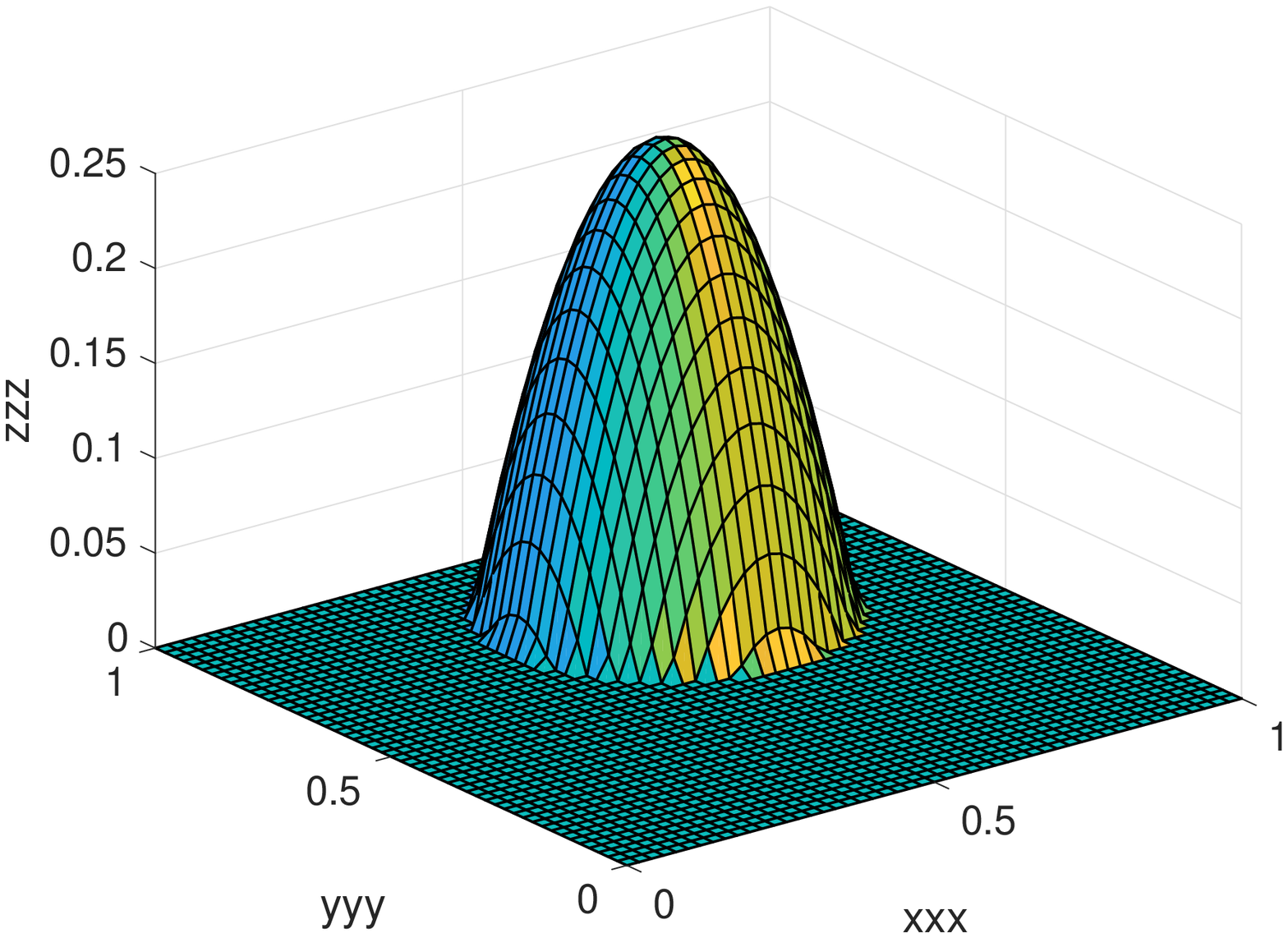}
\caption{Desired terminal state $z\in\cX_0$ for all $(x_1,x_2)\in X$ for TPBVP \er{eq:ex-TPBVP}.}
\label{fig:ex-TPBVP-a}
\end{center}
\vspace{-5mm}
\end{figure}%
\begin{figure}[h]
\begin{center}
\vspace{5mm}
\psfrag{xxx}{$x_1$}
\psfrag{yyy}{$x_2$}
\psfrag{zzz}{$\dot x(0)$}
\includegraphics[width=70mm]{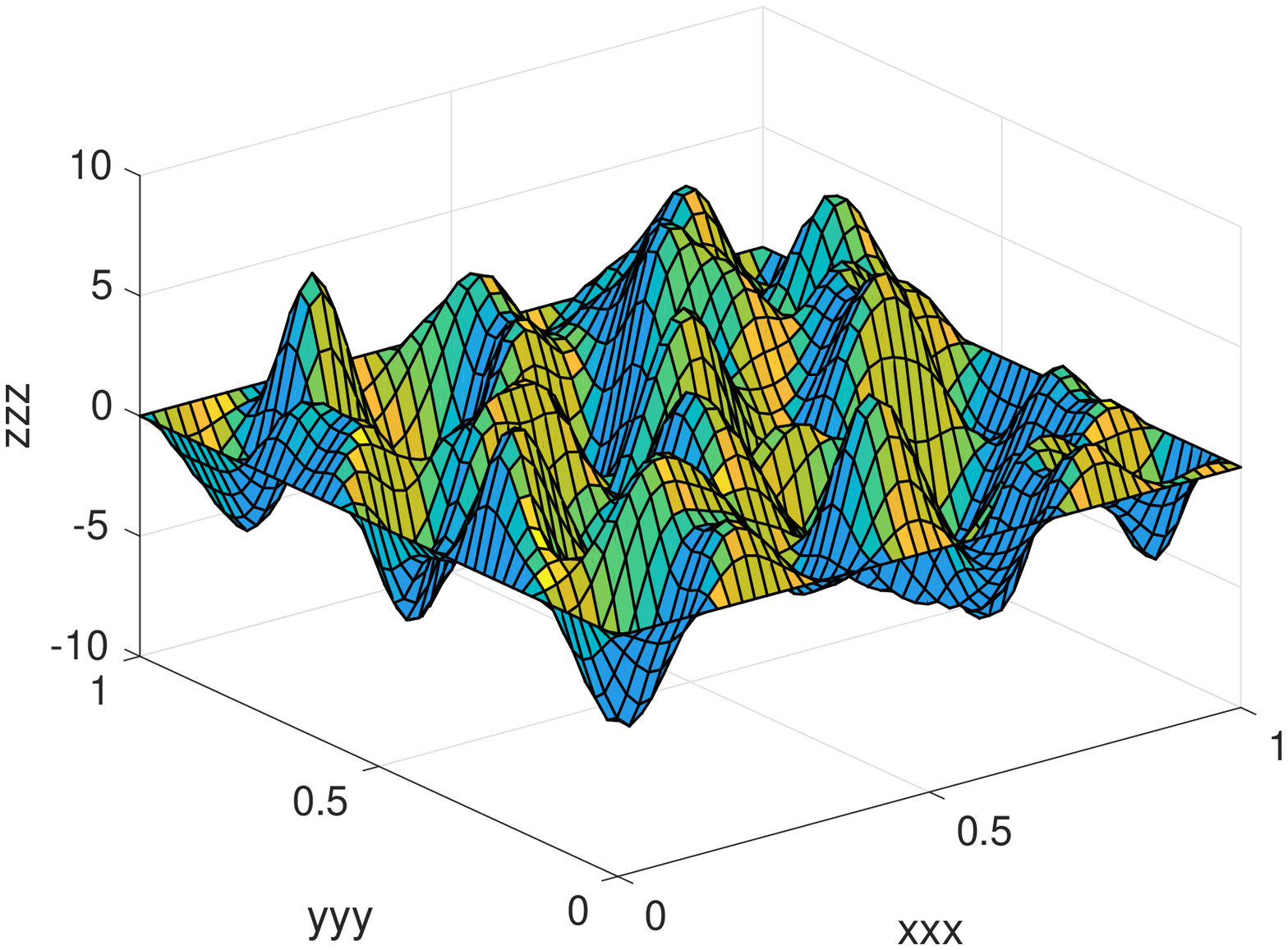}
\caption{Computed solution $\dot x(0)\in\cX_0$ of \er{eq:TPBVP-sol} for TPBVP \er{eq:ex-TPBVP}.}
\label{fig:ex-TPBVP-b}
\end{center}
\end{figure}

\begin{figure}[h]
\vspace{0mm}
\begin{center}
\psfrag{xxx}{$x_1$}
\psfrag{yyy}{$x_2$}
\psfrag{zzz}{$z$}
\includegraphics[width=100mm]{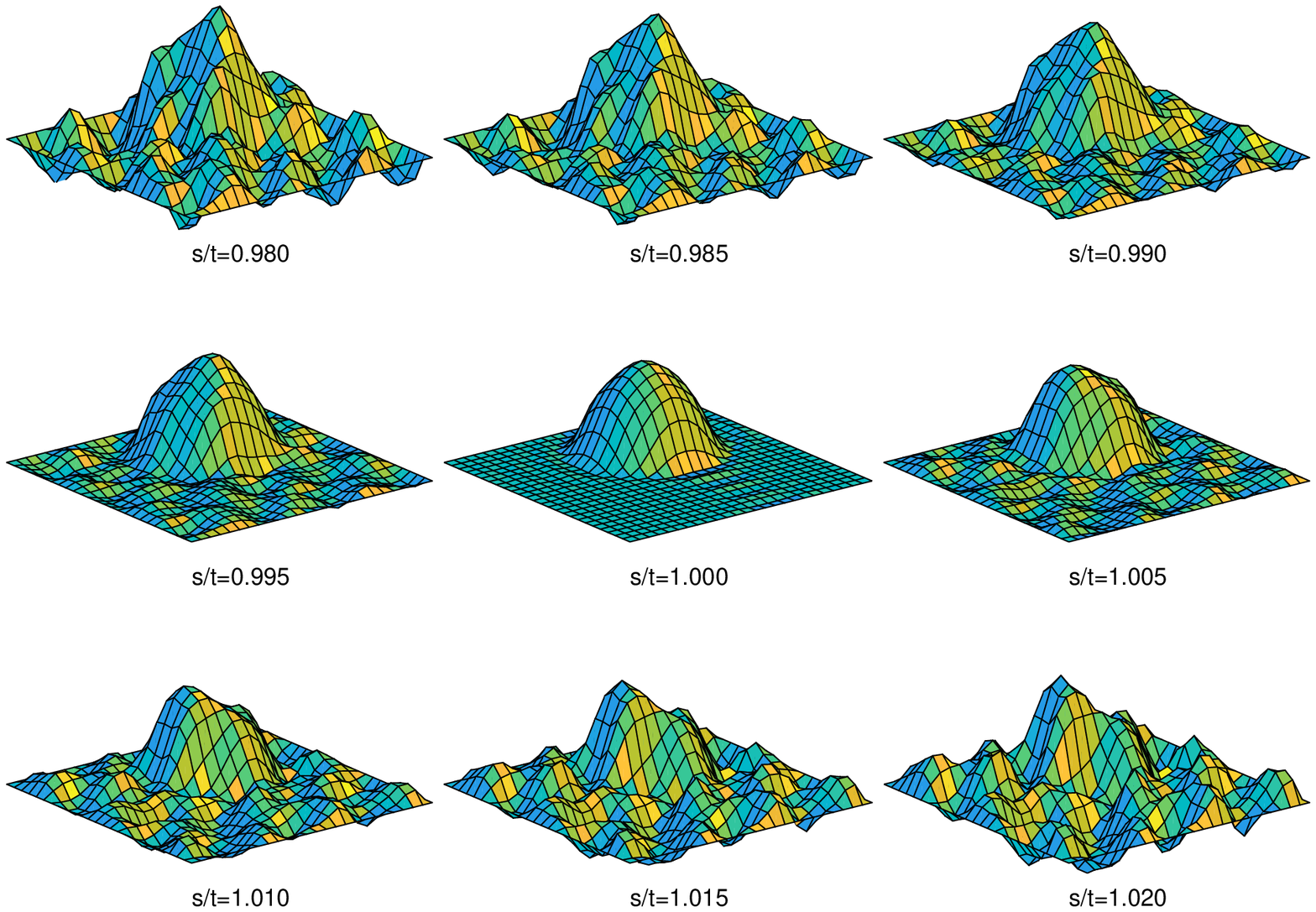}
\caption{Solution of \er{eq:wave} propagated forward from $x(0) = x$, $\dot x(0)$ as per \er{eq:TPBVP-sol}, to $s/t\in[0.98,1.02]$.}
\label{fig:evolution}
\end{center}
\end{figure}



\section{Conclusions}
\label{sec:conc}
A representation for the fundamental solution group for a class of wave equations is constructed via \rev{Hamilton's action} principle and an optimal control problem. In particular, solutions of a wave equation in the class of interest are identified as rendering a corresponding action functional stationary. By encapsulating this action functional in an optimal control problem, these solutions are expressed as the corresponding optimal dynamics involved. By employing a idempotent convolution kernel to equivalently represent the value of the optimal control problem, a prototype of an approximation of the fundamental solution group involved is obtained. However, as the action functional loses concavity (in this case) for longer time horizons, the prototype fundamental solution group is restricted to short time horizons. This restriction is subsequently relaxed via a relaxation of the optimal control problem to include stationary (rather than exclusively optimal) payoffs. The approximate fundamental solution group obtained, and its limit, are verified via the Trotter-Kato theorem to correspond to that of the class of approximating wave equations, and the exact wave equation respectively, of interest. They are applied in posing a TPBVP involving these equations, and finding its solution.\\[0mm]



\noindent
{\bf Acknowledgements.}
The authors acknowledge funding support provided by the US Air Force Office of Scientific Research. 




\bibliographystyle{IEEEtran}
\bibliography{wave-arXiv}


\end{document}